\documentclass[12pt, reqno]{amsart}
\usepackage{amsmath, amsthm, amscd, amsfonts, amssymb, graphicx, color}

\newtheorem{theorem}{Theorem}[section]
\newtheorem{lemma}[theorem]{Lemma}
\newtheorem{proposition}[theorem]{Proposition}
\newtheorem{corollary}[theorem]{Corollary}
\theoremstyle{definition}
\newtheorem{definition}[theorem]{Definition}

\theoremstyle{remark}
\newtheorem{remark}[theorem]{Remark}
\numberwithin{equation}{section}

\begin{document}
\setcounter{page}{1}
	
\title[Semi-Fredholm theory on Hilbert $C^*$-modules]{Semi-Fredholm theory on Hilbert $C^*$-modules }
\author[S. Ivkovi\'{c}]{Stefan Ivkovi\'{c}}
\email{\textcolor[rgb]{0.00,0.00,0.84}{stefan.iv10@outlook.com}}	
\address{The Mathematical institute of the Serbian Academy of Sciences and Arts}

\let\thefootnote\relax\footnote{Copyright 2019 The Mathematical Institute of the Serbian Academy of Sciences and Arts, Kneza Mihaila 36, 11000 Beograd, p.p. 367, Serbia, office@mi.sanu.ac.rs}
\email{\textcolor[rgb]{0.00,0.00,0.84}{office@mi.sanu.ac.rs}}
	

\begin{abstract}
In this paper we establish the semi-Fredholm theory on Hilbert $C^*$-modules as a continuation of Fredholm theory on Hilbert $C^*$-modules established by Mishchenko and Fomenko. We give a definition of a semi-Fredholm operator on Hilbert $C^*$-module and prove that these semi-Fredholm operators are those that are one-sided invertible modulo compact operators, that the set of proper semi-Fredholm operators is open and many other results that generalize their classical counterparts.
\end{abstract} \maketitle
\section{Introduction }
The Fredholm and semi-Fredholm theory on Hilbert and Banach spaces started by studying the certain integral equations which was done in the pioneering work by Fredholm in 1903 in \cite{F}. After that the abstract theory of Fredholm and semi-Fredholm operators on Banach spaces was during the time developed in numerous papers. Some recent results in the classical semi-Fredholm theory can be found in \cite{ZZDH}. Now, Fredholm theory on Hilbert $C^*$-modules as a generalization of Fredholm theory on Hilbert spaces was started by Mishchenko and Fomenko in \cite{MF}. They have elaborated the notion of a Fredholm operator on the standard module $H_{\mathcal{A}} $ and proved the generalization of the Atkinson theorem. Our aim is to study more general operators than the Fredholm ones, namely a generalization of semi-Fredholm operators. In this paper we give the definition of those and establish several properties as an analogue or a generalized version of the properties of the classical semi-Fredholm operators on Hilbert and Banach spaces. \\
Recall that if $H$ is a Hilbert space, then $F$ is a semi-Fredholm operator on $H$, denoted by $F \in \Phi_{\pm} (H) $ if $F \in B (H) $ and ${\rm ran} F$ is closed, that is if there exists a decomposition
$$ H = (\ker F)^{\perp} \oplus \ker F‎‎\stackrel{F}{\longrightarrow} {\rm ran} F \oplus ({\rm ran} F)^{\perp} =H $$ 
with respect to which $F$ has the matrix
$\left\lbrack
\begin{array}{ll}
F_{1} & 0 \\
0 & 0 \\
\end{array}
\right \rbrack,
$
where $ F_{1} $ is an isomorphism, and either
$$\dim \ker F < \infty \textrm{ or } \dim ({\rm ran} F)^{\perp} < \infty .$$ 
If $\dim \ker F < \infty ,$ then $F$ is called an upper semi-Fredholm operator on $H,$ denoted by $F \in \Phi_{+} (H) $ whereas if $\dim ({\rm ran} F)^{\perp} < \infty ,$ then $F$ is called 
% doatak 2 
a lower semi-Fredholm operator on $H $ denoted by $F \in \Phi_{-} (H) .$ If $F$ is both an upper and lower semi-Fredholm operator on $H$, then $F$ is said to be a Fredholm operator on $H$ denoted by $F \in \Phi (H) .$
In the case when $F \in \Phi (H) $ , the index of $F$ is defined as ${\rm index} F=\dim \ker F- \dim ({\rm ran} F)^{\perp}.$ Now, Hilbert $ C^{*} $-modules are a natural generalization of Hilbert spaces when the field of scalars is replaced by an arbitrary $ C^{*}$-algebra. Some recent results in the theory of Hilbert $C^*$-modules can be found in \cite{FMT}, \cite{HG}, \cite{L}, \cite{MSFC}, \cite{S3}. In \cite{MF} one considers a standard Hilbert $ C^{*}$-modul over a unital $ C^{*}$-algebra $\mathcal{A}$, denoted by $ H_{\mathcal{A}} $ and one defines an $\mathcal{A}$-Fredholm operator $F$ on $H_{\mathcal{A}} $ as generalization of a Fredholm operator on Hilbert space $H$ in the following way:
\cite[ Definition]{MF} A (bounded $\mathcal{A}$ linear) operator $F: H_{\mathcal{A}} \rightarrow H_{\mathcal{A}}$ is called $\mathcal{A}$-Fredholm if\\
1) it is adjointable;\\
2) there exists a decomposition of the domain $H_{\mathcal{A}}= M_{1} \tilde{\oplus} N_{1} ,$ and the range, $H_{\mathcal{A}}= M_{2} \tilde{\oplus} N_{2}$, where $M_{1},M_{2},N_{1},N_{2}$ are closed $\mathcal{A}$-modules and $N_{1},N_{2}$ have a finite number of generators, such that $F$ has the matrix from 
\begin{center}
	$\left\lbrack
	\begin{array}{ll}
	F_{1} & 0 \\
	0 & F_{4} \\
	\end{array}
	\right \rbrack
	$
\end{center}
with respect to these decompositions and $F_{1}:M_{1}\rightarrow M_{2}$ is an isomorphism.\\
It is then proved in \cite{MF} that some of the main results from the classical Fredholm theory on Hilbert spaces also hold when one considers this generalization of Fredholm operator on Hilbert $ C^{*}$-module. The idea in this paper was to go further in this direction, to give, in a similar way, a definition of semi-Fredholm operators on $H_{\mathcal{A}}  ,$ to investigate and prove generalized version in this setting of significantly many results from the classical semi-Fredholm theory on Hilbert and Banach spaces. \\
In the second section, inspired by \cite[ Definition]{MF}, we define upper and lower semi-$\mathcal{A}$-Fredholm operators on $ H_{\mathcal{A}}   .$ We say that F is an upper semi-$\mathcal{A}$-Fredholm operator on $H_{\mathcal{A}} $ if all the conditions in the definition above hold except that $N_{2}$ does not need to be finitely generated and we say that $F$ is a lower semi-$\mathcal{A}$-Fredholm operator on $H_{\mathcal{A}}$ if all the conditions in the definition above hold except that $N_{1}$ in this case does not need to be finitely generated. Then we show that the classes of upper and lower semi-$\mathcal{A}$-Fredholm operators on $H_{\mathcal{A}} $ denoted respectively by $\mathcal{M}\Phi_{+}(H_{\mathcal{A}})$ and $\mathcal{M}\Phi_{-}(H_{\mathcal{A}}) $ coincide with the inverse images of the sets of the left and right invertible elements respectively in the $ C^{*}$-algebra $ B^{a}(H_{\mathcal{A}}) / K (H_{\mathcal{A}}) $ under the quotient map, where $ B^{a}(H_{\mathcal{A}})$ denotes the $C^{*}$-algebra of all bounded, adjointable operators on $H_{\mathcal{A}} .$ Semi-$\mathcal{A}$-Fredholm operators on $H_{\mathcal{A}}$ have been considered in \cite{AH} and \cite{HA}. In \cite{AH} they define semi-$\mathcal{A}$-Fredholm operators to be those that are one-sided invertible modulo compact operators. However, in this paper we give another definition of semi-$\mathcal{A}$-Fredholm operators on $H_{\mathcal{A}}$ and then prove that these operators are exactly those that are one-sided invertible modulo compact operators. Moreover, we prove an analogue or generalized versions of main results in \cite[ Section 1.2] {ZZRD} and \cite[ Section 1.3] {ZZRD}, as well as some additional new results. We wish to remark that if one considers the classes of operators  $\mathcal{M}\Phi(H_{\mathcal{A}}),\mathcal{M}\Phi_{+}(H_{\mathcal{A}}),\mathcal{M}\Phi_{-}(H_{\mathcal{A}}) $ in the sense of Definition 2.1 in this paper, which are given in terms of the decompositions for given $ F,$ it is not obvious that $ \mathcal{M}\Phi(H_{\mathcal{A}})=\mathcal{M}\Phi_{+}(H_{\mathcal{A}}) \cap \mathcal{M}\Phi_{-}(H_{\mathcal{A}}) ,$ as we have in the classical semi-Fredholm theory on Hilbert and Banach spaces. This is due to the fact that these decompositions may not be unique for an $ F \in B^{a}(H_{\mathcal{A}}) $ whereas in the classical case one always considers the decomposition 
$$ H = (\ker F)^{\perp} \oplus \ker F‎‎\stackrel{F}{\longrightarrow} {\rm ran} F \oplus ({\rm ran} F)^{\perp} =H,$$ 
where $H$ is a Hilbert space and $ F \in B(H) .$ Key arguments in proving that $ \mathcal{M}\Phi(H_{\mathcal{A}})=\mathcal{M}\Phi_{+}(H_{\mathcal{A}}) \cap \mathcal{M}\Phi_{-}(H_{\mathcal{A}}) ,$ is the equivalence between $ \mathcal{M}\Phi_{+} $ property and the left invertibility in the $ C^{*}$-algebra $ B^{a}(H_{\mathcal{A}}) / K(H_{\mathcal{A}}) $ and the equivalence between $ \mathcal{M}\Phi_{-}$-property and the right invertibility in $ B^{a}(H_{\mathcal{A}}) / K(H_{\mathcal{A}}). $ \\
This is also the main argument in proving that $$\mathcal{M}\Phi_{+}(H_{\mathcal{A}}), \mathcal{M}\Phi_{-}(H_{\mathcal{A}}), 
\mathcal{M}\Phi_{+}(H_{\mathcal{A}}) \setminus \mathcal{M}\Phi(H_{\mathcal{A}}), \mathcal{M}\Phi_{-}(H_{\mathcal{A}}) \setminus \mathcal{M}\Phi(H_{\mathcal{A}})$$ 
are semigroups, as we have in the classical semi-Fredholm theory on Banach spaces. On the other hand, Lemma 2.16 and Lemma 2.17 are very important as well and they are also used in proving several other fundamental results in the paper, where we meet the challenge with non uniqueness of the decomposition. \\
In the third section in this paper we prove analogue results of \cite[ Lemma 1.4.1] {ZZRD} and \cite[ Lemma 1.4.2] {ZZRD}. More precisely, we give generalizations on Hilbert $C^{*}$-modules of the results from the classical semi-Fredholm theory on Banach spaces connected with the Schechter's characterization of $\Phi_{+}$ operators on Banach spaces given in \cite{S2} . \\
In the fourth section, we prove that $ \mathcal{M}\Phi_{+}(H_{\mathcal{A}}) \setminus \mathcal{M}\Phi(H_{\mathcal{A}}) $ and $ \mathcal{M}\Phi_{-}(H_{\mathcal{A}}) \setminus \mathcal{M}\Phi(H_{\mathcal{A}}) $ are open as an analogue of the well known result in the classical semi-Fredholm theory which states that the sets $\Phi_{+}(X) \setminus \Phi (X) $ and $\Phi_{+}(X) \setminus \Phi (X) $ are open, where $X$ is a Banach space, the result which was proved in \cite{S}. Here again Lemma 2.16 and Lemma 2.17 from the second chapter are one of the main arguments in the proof. Also, we prove an analogue version on $H_{\mathcal{A}} $ of \cite[ Corollary 1.6.10] {ZZRD} and \cite[ Lemma 1.6.11] {ZZRD}. \\
In the fifth section we give first a generalization on Hilbert $C^{*}$-modules of $\Phi_{+}^{-} $ and $\Phi_{-}^{+} $ operators on Hilbert spaces. A natural generalization on $ H_{\mathcal{A}} $ of the class $\Phi_{+}^{-} (X) $ as defined in \cite[ Definition 1.2.1] {ZZRD}, where $X$ is a Banach space would be the following:
Let $ F \in B^{a}(H_{\mathcal{A}})  .$ Then, $ F \in {\mathcal{M}\Phi_{+}^{-}}(H_{\mathcal{A}}) $ if there exists a decomposition
$$ H_{\mathcal{A}} = M_{1} \tilde \oplus {N_{1}}‎‎\stackrel{F}{\longrightarrow} M_{2} \tilde \oplus N_{2}= H_{\mathcal{A}} $$ 
with respect to which $F$ has the matrix
$ \left\lbrack
\begin{array}{ll}
F_{1} & 0 \\
0 & F_{4} \\
\end{array}
\right \rbrack
,$ where
$ F_{1} $ is an isomorphism, $N_{1}$ is finitely generated and $ N_{1} \preceq N_{2} ,$ that is $ N_{1} $ is isomorphic to a closed submodule of $ N_{2}  .$ If $ \mathcal{A}=\mathbb{C}  ,$ that is if $H_{\mathcal{A}} $ is an ordinary Hibert space, then this definition would concide with \cite[ Definition 1.2.1] {ZZRD} of the class of $ \Phi_{+}^{-} $ operators on Hilbert spaces. However, if $ X $ is a Hilbert or a Banach space, then $\Phi_{+}(X) \setminus \Phi(X) \subseteq \Phi_{+}^{-} (X) $ when we consider the \cite[ Definition 1.2.1] {ZZRD} of classes $ \Phi_{+}(X)   , \Phi(X) $ and $\Phi_{+}^{-} (X)  ,$ whereas if we consider this „generalized“ definition of the class $ \mathcal{M}\Phi_{+}^{-} (H_{\mathcal{A}})  ,$ it is not true in general that 
$$\mathcal{M}\Phi_{+}(H_{\mathcal{A}}) \setminus \mathcal{M}\Phi(H_{\mathcal{A}}) \subseteq \mathcal{M}\Phi_{+}^{-} (H_{\mathcal{A}}) .$$ 
This is due to the fact that given a finitely generated closed submodule $ N_{1} $ of $ H_{\mathcal{A}} $ and a countably, but not finitely generated closed submodule $ N_{2} $ of $ H_{\mathcal{A}}  ,$ it is not true in general that that $ N_{1} $ is isomorphic to a closed submodule of $ N_{2}   .$ Therefore, we define the class ${\mathcal{M}\Phi_{+}^{-}}^{\prime} (H_{\mathcal{A}}) $ to be the class of operators in $ B^{a}(H_{\mathcal{A}}) $ which satisfy the conditions of the suggested „generalized“ definition above of $ \mathcal{M}\Phi_{+}^{-} (H_{\mathcal{A}})   .$ Moreover, we define $ {\tilde {\mathcal{M}\Phi}}_{+}^{-} (H_{\mathcal{A}}) $ be the subclass of $ {\mathcal{M}\Phi_{+}^{-}}^{\prime} (H_{\mathcal{A}})$, where in the decomposition above we have in addition that $ N_{2}$ is finitely generated. 
Then we set
$$ \mathcal{M}\Phi_{+}^{-} (H_{\mathcal{A}}) = { {{\tilde {\mathcal{M}\Phi}}_{+}^{{-}^{\prime}}} (H_{\mathcal{A}})} \cup (\mathcal{M}\Phi_{+}(H_{\mathcal{A}}) \setminus \mathcal{M}\Phi(H_{\mathcal{A}})  ) .$$ 
In a similar way we define the classes $ {\mathcal{M}\Phi_{-}^{+}}^{\prime} (H_{\mathcal{A}}) , { {{\tilde {\mathcal{M}\Phi}}_{+}^{{-}^{\prime}}} (H_{\mathcal{A}})} $ and set
$$ \mathcal{M}\Phi_{-}^{+} (H_{\mathcal{A}}) = { {{\tilde {\mathcal{M}\Phi}}_{-}^{{+}^{\prime}}} (H_{\mathcal{A}})} \cup (\mathcal{M}\Phi_{-}(H_{\mathcal{A}}) \setminus \mathcal{M}\Phi(H_{\mathcal{A}})  ) .$$ 
We also then prove that $$ \tilde{\mathcal{M}\Phi}_{+}^{-}(H_{\mathcal{A}}) = {\mathcal{M}\Phi_{+}^{-}}^{\prime} (H_{\mathcal{A}}) \cap \mathcal{M}\Phi(H_{\mathcal{A}}), \tilde{\mathcal{M}\Phi}_{-}^{+}(H_{\mathcal{A}}) = {\mathcal{M}\Phi_{-}^{+}}^{\prime} (H_{\mathcal{A}}) \cap \mathcal{M}\Phi(H_{\mathcal{A}})  .$$ In addition we show that if $ F\in \tilde{\mathcal{M}\Phi}_{+}^{-}(H_{\mathcal{A}}) $ and 
$$H_{\mathcal{A}} = M_{1} \tilde \oplus {N_{1}}‎‎\stackrel{F}{\longrightarrow} M_{2} \tilde \oplus N_{2}= H_{\mathcal{A}} $$ 
is any decomposition with respect to which $ F $ has the matrix
$ \left\lbrack
\begin{array}{ll}
F_{1} & 0 \\
0 & F_{4} \\
\end{array}
\right \rbrack
,$
where $ F_{1} $ is an isomorphism and $ N_{1},N_{2} $ are finitely generated, then if $K(A)$ satisfies cancellation property, we have $ N_{1} \preceq N_{2}   .$ Similar conclusion yields for operators in the class $ \tilde{\mathcal{M}\Phi}_{-}^{+}(H_{\mathcal{A}}) $ only in this case $ N_{2} \preceq N_{1}.$ In addition, we show that the classes 
$$ \tilde{\mathcal{M}\Phi}_{+}^{-}(H_{\mathcal{A}}) , \tilde{\mathcal{M}\Phi_{-}^{+}}(H_{\mathcal{A}}), {\mathcal{M}\Phi_{+}^{-}}^{\prime} (H_{\mathcal{A}}) , {\mathcal{M}\Phi_{-}^{+}}^{\prime} (H_{\mathcal{A}}) $$ 
are open. In the rest of this section we work with these classes of operators and prove analogy or generalized versions of almost all results in \cite[ Section 1.9] {ZZRD}.\\
The generalized versions on $H_{\mathcal{A}} $ of the results from the classical semi-Fredholm theory on Banach and Hilbert spaces, which are presented here in this paper, demand different proofs from the proofs in the classical case, however the techniques used in these proofs are to a certain extent inspired by the techniques used in the proofs of some of the results in \cite{MF}. Moreover, the first section and especially the fourth section, where we introduce new, additional classes of operators as $ {\tilde{\mathcal{M}\Phi}_{+}^{-}}(H_{\mathcal{A}}) , {\tilde{\mathcal{M}\Phi}_{-}^{+}}(H_{\mathcal{A}}), {\mathcal{M}\Phi_{+}^{-}}^{\prime} (H_{\mathcal{A}}) $ and $ {\mathcal{M}\Phi_{-}^{+}}^{\prime} (H_{\mathcal{A}}) $ contain some new results and not just generalizations on $ H_{\mathcal{A}} $ of results from the classical semi-Fredholm theory on Banach spaces. 
%\section{\textbf{Introduction}}
\section{ Semi-{$\mathcal{A}$}-Fredholm operators on $H_{\mathcal{A}}$}
In this section we define semi-${\mathcal{A}}$-Fredholm operators on the standard module $H_{\mathcal{A}} $ and prove some of the main properties and results concerning these operators. Throughout this paper we let $\mathcal{A}$ be a unital $C^{*}$-algebra, $H_{\mathcal{A}}$ be a standard module over $\mathcal{A}$ and we let $B^{a}(H_{\mathcal{A}})$ denote the set of all bounded , adjointable operators on $H_{\mathcal{A}}.$ According to \cite[ Definition 1.4.1] {MT}, we say that a Hilbert $C^*$-module $M$ over $\mathcal{A}$ is finitely generated if there exists a finite set $ \lbrace x_{i} \rbrace \subseteq M $  such that $M $ equals the linear span (over $\textbf{C}$ and $\mathcal{A} $) of this set.
\begin{definition} 
Let $F \in B^{a}(H_{\mathcal{A}}).$ We say that $F $ is an upper semi-{$\mathcal{A}$}-Fredholm operator if there exists a decomposition $$H_{\mathcal{A}} = M_{1} \tilde \oplus {N_{1}}‎‎\stackrel{F}{\longrightarrow} M_{2} \tilde \oplus N_{2}= H_{\mathcal{A}} $$ with respect to which $F$ has the matrix\\

\begin{center}
	$\left\lbrack
	\begin{array}{ll}
	F_{1} & 0 \\
	0 & F_{4} \\
	\end{array}
	\right \rbrack,
	$
\end{center}
where $F_{1}$ is an isomorphism $M_{1},M_{2},N_{1},N_{2}$ are closed submodules of $H_{\mathcal{A}} $ and $N_{1}$ is finitely generated. Similarly, we say that $F$ is a lower semi-{$\mathcal{A}$}-Fredholm operator if all the above conditions hold except that in this case we assume that $N_{2}$ ( and not $N_{1}$ ) is finitely generated.
\end{definition}
Set
\begin{center}
$\mathcal{M}\Phi_{+}(H_{\mathcal{A}})=\lbrace F \in B^{a}(H_{\mathcal{A}}) \mid F $ is upper semi-{$\mathcal{A}$}-Fredholm $\rbrace ,$	
\end{center}
\begin{center}
$\mathcal{M}\Phi_{-}(H_{\mathcal{A}})=\lbrace F \in B^{a}(H_{\mathcal{A}}) \mid F $ is lower semi-{$\mathcal{A}$}-Fredholm $\rbrace ,$	
\end{center}
$\mathcal{M}\Phi(H_{\mathcal{A}})=\lbrace F \in B^{a}(H_{\mathcal{A}}) \mid F $ is $\mathcal{A}$-Fredholm operator on $H_{\mathcal{A}}\rbrace .$ Then obviously $\mathcal{M}\Phi (H_{\mathcal{A}}) \subseteq \mathcal{M}\Phi_{+}(H_{\mathcal{A}}) \cap \mathcal{M}\Phi_{-}(H_{\mathcal{A}})$ . We are going to show later in this section that actually "=" holds. 

Notice that if $M,N$ are two arbitrary Hilbert modules $C^{*}$-modules, the definition above could be generalized to the classes $\mathcal{M}\Phi_{+}(M,N)$ and $\mathcal{M}\Phi_{-}(M,N)$.\\
Recall that by \cite[ Definition 2.7.8]{MT}, originally given in \cite{MF}, when $F \in \mathcal{M}\Phi(H_{\mathcal{A}})     $ and 
$$ H_{\mathcal{A}} = M_{1} \tilde \oplus {N_{1}}‎‎\stackrel{F}{\longrightarrow} M_{2} \tilde \oplus N_{2}= H_{\mathcal{A}} $$
is an $ \mathcal{M}\Phi    $ decomposition for $  F   $, then the index of $F$ is definited by index $ F=[N_{1}]-[N_{2}] \in K(\mathcal{A})    $ where $[N_{1}]    $ and $ [N_{2}] $ denote  the isomorphism classes of $ N_{1}    $ and $ N_{2} $ respectively. „By \cite[ Definition 2.7.9]{MT}, the index is well defined and does not depend on the choice of $\mathcal{M}\Phi$ decomposition for $F$. As regards the $K$-group $K (\mathcal{A})$, it is worth mentioning that it is not true in general that [M]=[N] implies that $ M \cong N    $ for two finitely generated submodules $M, N$ of $ H_{\mathcal{A}}    $. If $K (\mathcal{A})$ satisfies the property that [N]=[M] implies that $ N \cong M     $ for any two finitely generated, closed submodules $M, N$ of $ H_{\mathcal{A}}     $, then $K (\mathcal{A})   $ is said to satisfy "the cancellation property", see \cite[Section 6.2] {W}.
\begin{theorem}
Let $F \in B^{a}(H_{\mathcal{A}}).$ The following statements are equivalent\\
	1)	$F \in \mathcal{M}\Phi_{+}(H_{\mathcal{A}}) $\\
	2)	There exists $D \in B^{a}(H_{\mathcal{A}}) $ such that $DF=I+K $ for some $K \in K(H_{\mathcal{A}})$
\end{theorem}
\begin{proof}
2)$\Rightarrow $ 1) If 2) holds, then $DF\in \mathcal{M}\Phi (H_{\mathcal{A}})$ by \cite[ Lemma 2.7.12] {MT}. Let $H_{\mathcal{A}} = M_{1} \tilde \oplus {N_{1}}‎‎\stackrel{DF}{\longrightarrow} M_{2} \tilde \oplus N_{2}= H_{\mathcal{A}} $ be a decomposition w.r.t which DF has the matrix
\begin{center}
	$	\left\lbrack
	\begin{array}{ll}
	(DF)_{1} & 0 \\
	0 & (DF)_{4} \\
	\end{array}
	\right \rbrack,
	$
\end{center}
where $(DF)_{1}$ is an isomorphism and $N_{1},N_{2}$ are finitely generated. 
We wish to show that $ F(M_{1}) $ is closed and we will do it by showing that $F_{\mid_{M_{1}}} $ is bounded below. Suppose that this is not the case. Then there exists a sequence $\lbrace x_{n} \rbrace \subseteq M_{1} $ s.t. $|| x_{n} ||=1 $ for all $n$ and $ F x_{n} \rightarrow 0 $ as $n\rightarrow  \infty.$ Since $ D $ is bounded, we must have that $DF x_{n} \rightarrow 0 $ as $n\rightarrow + \infty .$ But this would mean that $DF$ is not bounded below on $M_{1} $ as $ || x_{n} ||=1 $ for all $n.$ This is a contradiction since $DF_{\mid_{M_{1}}} $ is an isomorphism. Hence we must have that $F$ is bounded below on $ M_{1} $ which means that $ F(M_{1}) $ is closed.\\
Now, by \cite[ Theorem 2.7.6] {MT}, the result which was originally proved in \cite{T}, we may assume that $M_{1}$ is orthogonally complementable in $H_{\mathcal{A}}$. Hence $F_{{\mid}_{{M}_{1}}}$ is adjointable, so by \cite[ Theorem 2.3.3] {MT}, which was originally proved in \cite{M} , ${\rm ran} F_{{\mid}_{{M}_{1}}}$ is orthogonally complementable in $H_{\mathcal{A}}$.\\
Hence $H_{\mathcal{A}}=F(M_{1}) \oplus F(M_{1})^{\perp} .$ With respect to the decomposition 
$$ H_{\mathcal{A}}=M_{1} \tilde{\oplus} {N_{1}}‎‎\stackrel{F}{\longrightarrow} F(M_{1}) \oplus F(M_{1})^{\perp}=H_{\mathcal{A}} ,$$
$F$ has the matrix
$\left\lbrack
\begin{array}{ll}
F_{1} & F_{2} \\
0 & F_{4} \\
\end{array}
\right \rbrack,
$
where $F_{1}$ is an isomorphism. If we let 
\begin{center}
	$
	U=\left\lbrack
	\begin{array}{ll}
	1 &- {F_{1}}^{-1}{ F_{2}} \\
	0 & 1 \\
	\end{array}
	\right \rbrack 
	$	
\end{center}
with respect to the decomposition $$H_{\mathcal{A}} = M_{1} \tilde \oplus {N_{1}}‎‎\stackrel{U}{\longrightarrow} M_{1} \tilde \oplus N_{1}= H_{\mathcal{A}}  ,$$ then $ U $ is an isomorphism and with respect to the decomposition 
$$ H_{\mathcal{A}} =U( M_{1}) \tilde \oplus U({N_{1})}‎‎\stackrel{F}{\longrightarrow} F(M_{1}) \tilde \oplus F(M_{1})^{\perp}= H_{\mathcal{A}} $$ 
$F$ has the matrix
$\left\lbrack
\begin{array}{ll}
F_{1} & 0 \\
0 & \tilde{F}_{4} \\
\end{array}
\right \rbrack .
$
Since $ N_{1} $ is finitely generated, $ U(N_{1}) $ is finitely generated also, hence $ F \in \mathcal{M}\Phi_{+}(H_{\mathcal{A}})  .$\\
$1)\Rightarrow 2)$ \\
Let 
$$H_{\mathcal{A}} = M_{1} \tilde \oplus {N_{1}}‎‎\stackrel{F}{\longrightarrow} M_{2} \tilde \oplus N_{2}= H_{\mathcal{A}} $$ 
be a decomposition with respect to which $F$ has the matrix
$\left\lbrack
\begin{array}{ll}
F_{1} & 0 \\
0 & F_{4} \\
\end{array}
\right \rbrack,
$ 
where $ F_{1} $ is an isomorphism and $N_{1} $ is finitely generated. Since $ N_{1} $ is finitely generated, it is orthogonally complementable in $H_{\mathcal{A}} $ by \cite[ Lemma 2.3.7] {MT} which was originally proved in \cite{M}. Then, by the proof of \cite[ Theorem 2.7.6] {MT} , we can deduce that $F_{{\mid}_{N_{{1}}^{\bot}}} $ is an isomorphism onto $F{{({N_{1}}^{\bot})}} $.
Now, $F{{({N_{1}}^{\bot})}}={\rm ran} FP_{{{N_{1}}^{\bot}}}$, where $ P_{{{N_{1}}^{\bot}}} $ denotes the orthogonal projection onto ${N_{1}}^{\bot} $. Since $FP_{{{N_{1}}^{\bot}}} \in B^{a}(H_{\mathcal{A}}) $ and $F{{({N_{1}}^{\bot})}} $ is closed being isomorphic to ${N_{1}}^{\bot} ,$ by \cite[ Theorem 2.3.3] {MT}, it follows that $F{{({N_{1}}^{\bot})}} $ is orthogonally complementable. With respect to the decomposition 
$$H_{\mathcal{A}} = {N_{1}}^{\bot} \oplus {N_{1}}‎‎\stackrel{F}{\longrightarrow} F({N_{1}}^{\bot} ) \oplus F({N_{1}}^{\bot} )^{\bot}=H_{\mathcal{A}} $$
$ F $ has the matrix
$	\left\lbrack
\begin{array}{ll}
{\tilde F_{1}} & {\tilde F_{2}} \\
0 & {\tilde F_{4}} \\
\end{array}
\right \rbrack,
$ 
where $ {\tilde F_{1}} $ is an isomorphism. Clearly $\tilde F_{1}, \tilde F_{2}$ and $ \tilde F_{4} $ are then adjointable.\\
Let $D$ be the operator which has the matrix
$	\left\lbrack
\begin{array}{ll}
{\tilde F_{1}}^{-1} & 0 \\
0 & 0 \\
\end{array}
\right \rbrack
$ 
with respect to the decomposition
$$H_{\mathcal{A}} =F({N_{1}}^{\bot} ) \oplus {F({N_{1}}^{\bot} )^{\bot}}‎‎\stackrel{D}{\longrightarrow} {N_{1}}^{\bot} \oplus {N_{1}} =H_{\mathcal{A}} $$
Then $D \in B^{a}(H_{\mathcal{A}}) $ and
$
DF=\left\lbrack
\begin{array}{ll}
1 & {\tilde F_{1}}^{-1}{\tilde F_{2}} \\
0 & 0 \\
\end{array}
\right \rbrack 
$
with respect to the decomposition 
$$H_{\mathcal{A}} ={N_{1}}^{\bot} \oplus {N_{1}}‎‎\stackrel{DF}{\longrightarrow} {N_{1}}^{\bot} \oplus {N_{1}} =H_{\mathcal{A}} . $$
Let 
$K=\left\lbrack
\begin{array}{ll}
0 & {\tilde F_{1}}^{-1}\tilde{F_{2}} \\
0 & -1 \\
\end{array}
\right \rbrack 
$ 
with respect to the same decomposition. Since $ N_{1} $ is finitely generated, we have $K \in K(H_{\mathcal{A}}) $. Moreover, $DF=I+K.$
\end{proof} 
\begin{theorem}
Let $ D \in B^{a}(H_{\mathcal{A}}) .$ Then the following statements are equivalent:\\
	1)	$ D \in \mathcal{M}\Phi_{-}(H_{\mathcal{A}}) $\\
	2)	There exist $F \in B^{a}(H_{\mathcal{A}}), K \in K(H_{\mathcal{A}}) $ s.t. $DF=I+K $
\end{theorem}
\begin{proof}
$2) \Rightarrow 1) $\\
Let $$ H_{\mathcal{A}} = M_{1} \tilde \oplus {N_{1}}^{^{‎‎\stackrel{I+K}{\longrightarrow}}} M_{2} \tilde \oplus N_{2}= H_{\mathcal{A}} . $$ be an $\mathcal{M}\Phi$ decomposition for $ I+K .$
As in the proof of Theorem 2.2, we deduce that $ F(M_{1}) $ is closed and orthogonally complementable in $H_{\mathcal{A}} .$\\
With respect to the decomposition 
$$ H_{\mathcal{A}}=F(M_{1}) \tilde{\oplus} {F({M_{1}})^{\perp} }‎‎\stackrel{D}{\longrightarrow} M_{2} \oplus N_{2}=H_{\mathcal{A}} ,$$
$D$ has the matrix 
$\left\lbrack
\begin{array}{ll}
D_{1} & D_{2} \\
0 & D_{4} \\
\end{array}
\right \rbrack,
$
where $ D_{1} $ is an isomorphism, as in the proof of Theorem 2.2, part $2) \Rightarrow 1)  ,$ we deduce thet $ D $ has the matrix 
$\left\lbrack
\begin{array}{ll}
D_{1} & 0 \\
0 & \tilde D_{4} \\
\end{array}
\right \rbrack
$
with respect to the decomposition
$$ H_{\mathcal{A}}=\tilde U(F(M_{1})) \tilde{\oplus} \tilde{U}({F({M_{1}})^{\perp} )}‎‎\stackrel{D}{\longrightarrow} M_{2} \oplus N_{2}=H_{\mathcal{A}},$$
where $ \tilde{U} $ is an isomorpism. Since $ N_{2} $ is finitely generated, it follows that $ D \in \mathcal{M}\Phi_{-}(H_{\mathcal{A}}) .$\\
$ 1) \Rightarrow 2) $\\
Let 
$$ H_{\mathcal{A}} = M_{1}^{\prime} \tilde \oplus {N_{1}^{\prime}}‎‎\stackrel{D^{\prime}}{\longrightarrow} M_{2}^{\prime} \tilde \oplus N_{2}^{\prime}= H_{\mathcal{A}} $$ 
be an $\mathcal{M}\Phi_{-} $ decomposition for $ D $ ( so that $ N_{2}^{\prime} $ is finitely generated ). Since $ N_{2}^{\prime} $ is finitely generated, it is orthogonally complementable by \cite[ Lemma 2.3.7] {MT}.
Now, since
$$ H_{\mathcal{A}} = M_{2}^{\prime} \tilde \oplus {N_{2}^{\prime}}= {N_{2}^{\prime}}^{\perp} \tilde \oplus N_{2}^{\prime},$$ 
we have that $P_{{{{N_{2}^{\prime}}^{\bot}}}_{{\mid}_{M_{2}^{\prime}}}} $ is an isomorphism from $ M_{2}^{\prime} $ onto $ {N_{2}^{\prime}}^{\perp}$, where $P_{{N_{2}^{\prime}}^{\perp}} $ denotes the orthogonal projection onto ${N_{2}^{\prime}}^{\perp} .$ Since $D$ has the matrix
$\left\lbrack
\begin{array}{ll}
D_{1} & 0 \\
0 & D_{4} \\
\end{array}
\right \rbrack
$ 
with respect to the decomposition 
$$ H_{\mathcal{A}} = M_{1}^{\prime} \tilde \oplus {N_{1}^{\prime}}‎‎\stackrel{D}{\longrightarrow} M_{2}^{\prime} \tilde \oplus N_{2}^{\prime}= H_{\mathcal{A}},$$
where $D_{1} $ is an isomorphism, it follows that $D^{-1}(N_{2}^{\prime})=N_{1}^{\prime} $, $\ker P_{{N_{2}^{\prime}}^{\perp}} D=N_{1}^{\prime} $ and ${\rm ran} P_{{N_{2}^{\prime}}^{\perp}} D= P_{{N_{2}^{\prime}}^{\perp}} (M_{2}^{\prime}) = {N_{2}^{\prime}}^{\perp}$ which is closed. By \cite[ Theorem 2.3.3] {MT}, $ \ker P_{{N_{2}^{\prime}}^{\perp}} D=N_{1}^{\prime} $ is orthogonally complementable, so $H_{\mathcal{A}} = {N_{1}^{\prime}}^{\perp} \oplus {N_{1}^{\prime}}.$ Hence $\sqcap_{{M_{1}^{\prime}}_{{\mid}_{{N_{1}^{\prime}}^{\perp} }}} $ is an isomorphism from ${N_{1}^{\prime}}^{\perp} $ onto $ M_{1}^{\prime} $, where $\sqcap_{M_{1}^{\prime}} $ denotes the projection onto $M_{1}^{\prime} $ along $ N_{1}^{\prime}$. Therefore, $ P_{{N_{2}^{\prime}}^{\perp}} D \sqcap_{{M_{1}^{\prime}}_{{\mid}_{{N_{1}^{\prime}}^{\perp} }}} $ is an isomorphism from ${N_{1}^{\prime}}^{\perp} $ onto ${N_{2}^{\prime}}^{\perp} .$ But since $D^{-1}(N_{2}^{\prime}) = N_{1}^{\prime} $ and $H_{\mathcal{A}} = {M_{1}^{\prime}} \tilde{\oplus} N_{1}^{\prime} ,$ it follows that 
$$P_{{N_{2}^{\prime}}^{\perp}} D_{\mid_{{{N_{1}^{\prime}}^{\perp}}}}=P_{{N_{2}^{\prime}}^{\perp}} D \sqcap_{{M_{1}^{\prime}}_{{\mid}_{{N_{1}^{\prime}}^{\perp} }}} .$$ 
Hence $P_{{N_{2}^{\prime}}^{\perp}} D_{\mid_{{{N_{1}^{\prime}}^{\perp}}}} $ 
is an isomorphism from ${N_{1}^{\prime}}^{\perp} $ onto ${N_{2}^{\prime}}^{\perp} ,$ being a composition of isomorphisms, so with respect to the decomposition
$$ H_{\mathcal{A}} = {N_{1}^{\prime}}^{\perp} \oplus {N_{1}^{\prime}}‎‎\stackrel{D}{\longrightarrow} {N_{2}^{\prime}}^{\perp} \oplus {N_{2}^{\prime}} = H_{\mathcal{A}} , $$
$D$ has the matrix
$\left\lbrack
\begin{array}{ll}
\tilde{D_{1}} & 0 \\
\tilde{D_{3}} &\tilde{D_{4}} \\
\end{array}
\right \rbrack,
$ 
where $ \tilde{D_{1}} $ is an isomorphism. \\
Let 
$F=\left\lbrack
\begin{array}{ll}
(\tilde{D_{1}})^{-1} & 0 \\
0 & 0 \\
\end{array}
\right \rbrack
$ 
with respect to the decomposition
$$ H_{\mathcal{A}} = {N_{2}^{\prime}}^{\perp} \oplus {N_{2}^{\prime}}‎‎\stackrel{F}{\longrightarrow} {N_{1}^{\prime}}^{\perp} \oplus {N_{1}^{\prime}} = H_{\mathcal{A}} .$$ 
Then $F \in B^{a}(H_{\mathcal{A}}) $ and
$DF=\left\lbrack
\begin{array}{ll}
1 & 0\\
\tilde D_{3} \tilde D_{1}^{-1} & 0 \\
\end{array}
\right \rbrack
$ 
with respect to the decomposition
$$ H_{\mathcal{A}} = {N_{2}^{\prime}}^{\perp} \oplus {N_{2}^{\prime}}‎‎\stackrel{DF}{\longrightarrow} {N_{2}^{\prime}}^{\perp} \oplus {N_{2}^{\prime}} = H_{\mathcal{A}} .$$ 
Since $ N_{2}^{\prime} $ is finitely generated, it follows that if we let the operator\\
$K=\left\lbrack
\begin{array}{ll}
0 & 0\\
D_{3}D_{1}^{-1} & -1 \\
\end{array}
\right \rbrack
$ 
w.r.t the decomposition above , the $K \in (H_{\mathcal{A}}) $. 
Moreover $DF=I+K.$
\end{proof}
Recall that $B^{a}(H_{\mathcal{A}}) $ is a $C^{*}$-algebra and $K(H_{\mathcal{A}}) $ is a closed two sided ideal in $B^{a}(H_{\mathcal{A}}) .$ Hence $B^{a}(H_{\mathcal{A}}) / K(H_{\mathcal{A}}) $ is also $C^{*}$-algebra, equipped with the quotient norm. We will call this algebra the "Calkin" algebra.
\begin{corollary}
$ \mathcal{M}\Phi (H_{\mathcal{A}})=\mathcal{M}\Phi_{+}(H_{\mathcal{A}}) \cap \mathcal{M}\Phi_{-}(H_{\mathcal{A}})	$
\end{corollary}
\begin{proof}
It suffices to show $ "\supseteq" .$ By Theorem 2.2, $\mathcal{M}\Phi_{+}(H_{\mathcal{A}}) $ consists of all elements that are left invertible in the "Calkin" algebra, whereas $\mathcal{M}\Phi_{-}(H_{\mathcal{A}}) $ consists of all elements that are right invertible in the "Calkin" algebra by Theorem 2.3. Now by \cite[ Theorem 2.7.14] {MT} and also by the proof of \cite[ Lemma 2.7.15] {MT}, we have that $ \mathcal{M}\Phi (H_{\mathcal{A}}) $ consists of all elements that are invertible in the "Calkin" algebra. The corollary follows. 
\end{proof}
\begin{corollary}
$ \mathcal{M}\Phi_{+}(H_{\mathcal{A}}) $ and $ \mathcal{M}\Phi_{-}(H_{\mathcal{A}}) $ are semigroups under multiplication. 
\end{corollary}
\begin{proof}
Follows directly from Theorem 2.2 and Theorem 2.3, as $ \mathcal{M}\Phi_{+}(H_{\mathcal{A}}) $ consists of all elements that are left invertible  in the "Calkin" algebra whereas $ \mathcal{M}\Phi_{+}(H_{\mathcal{A}}) $ consists of all elements that are right invertible in the "Calkin" algebra. 
\end{proof}
\begin{corollary}
Let $ F,D \in B^{a}(H_{\mathcal{A}}) .$ If $DF \in \mathcal{M}\Phi_{+}(H_{\mathcal{A}}) ,$ then $F \in \mathcal{M}\Phi_{+}(H_{\mathcal{A}}) .$ If $ DF \in \mathcal{M}\Phi_{-}(H_{\mathcal{A}}) ,$ then $D \in \mathcal{M}\Phi_{-}(H_{\mathcal{A}}) .$\\
\end{corollary}
\begin{proof}
Suppose that  $DF \in \mathcal{M}\Phi_{+}(H_{\mathcal{A}}).$ By Theorem 2.2 there exists some $  C \in B^{a} (H_{\mathcal{A}}), K \in K(H)     $ s.t. $CDF=I+K.$  Again, by Theorem 2.2 it follows that $F \in \mathcal{M}\Phi_{+}(H_{\mathcal{A}}).$ The proof of the second statement of Corollary 2.6 is similar.\\	
\end{proof}	
\begin{corollary}
Let $F,D \in B^{a}(H_{\mathcal{A}})  .$ If $DF \in \mathcal{M}\Phi_{+}(H_{\mathcal{A}}) $ and $ F\in \mathcal{M}\Phi(H_{\mathcal{A}}) ,$ then $D \in \mathcal{M}\Phi_{+}(H_{\mathcal{A}})  .$ If $ DF \in \mathcal{M}\Phi_{-}(H_{\mathcal{A}}) $ and $ D \in \mathcal{M}\Phi(H_{\mathcal{A}})  ,$ then $F \in \mathcal{M}\Phi_{-}(H_{\mathcal{A}})  .$
\end{corollary}
\begin{proof}
Suppose that $DF \in \mathcal{M}\Phi_{+}(H_{\mathcal{A}})$ and $ F \in \mathcal{M}\Phi(H_{\mathcal{A}}).$
By Theorem 2.2 there exist some $  C \in B^{a} (H_{\mathcal{A}}), K \in K(H_{\mathcal{A}})    $ s.t. $ CDF=I+K,$ as $ DF \in \mathcal{M}\Phi_{+}(H_{\mathcal{A}})     $ by assumption. Moreover, since $F \in \mathcal{M}\Phi(H_{\mathcal{A}})      ,$ by the proof of \cite[ Lemma 2.7.15] {MT} there exist some $F^{\prime}  \in B^{a} (H_{\mathcal{A}}), K^{\prime} \in K(H_{\mathcal{A}}) $ s.t. $FF^{\prime}=I+K^{\prime}.$ Hence $CDFF^{\prime}=(CDF)F^{\prime}=(I+K)F^{\prime}=F^{\prime}+KF^{\prime}$ and $ CDFF^{\prime}=CD(FF^{\prime})=CD(I+K^{\prime})=CD+CDK ^{\prime}   .$  Therefore, 
$FF^{\prime}+FKF^{\prime}=FCD+FCDK^{\prime} .$ So $FCD=FF^{\prime}+FKF^{\prime}-FCDK^{\prime}=I+K^{\prime}+FKF^{\prime}-FCDK^{\prime}      .$ Since $ K^{\prime} +FKF^{\prime}-FCDK^{\prime} \in K(H_{\mathcal{A}}),$ by Theorem 2.2 it follows that $ D \in \mathcal{M}\Phi(H_{\mathcal{A}})     .$ The proof of the second statement of Corollary 2.7 is similar.
\end{proof}
\begin{corollary}
Let $ F,D \in B^{a}(H_{\mathcal{A}})  .$ If $D \in \mathcal{M}\Phi_{+}(H_{\mathcal{A}}) $ and $DF \in \mathcal{M}\Phi(H_{\mathcal{A}})  ,$ then $D \in \mathcal{M}\Phi(H_{\mathcal{A}})  .$ If $F \in \mathcal{M}\Phi_{-}(H_{\mathcal{A}}) $ and $DF \in \mathcal{M}\Phi(H_{\mathcal{A}}) ,$ then $F \in \mathcal{M}\Phi(H_{\mathcal{A}}) .$
\end{corollary}
\begin{proof}
Let $D \in \mathcal{M}\Phi_{+}(H_{\mathcal{A}}) $ and $DF \in \mathcal{M}\Phi(H_{\mathcal{A}})  .$ Since $ DF \in \mathcal{M}\Phi(H_{\mathcal{A}})     ,$ by the proof of \cite[Lemma 2.7.15] {MT} there exists some $ C \in B^{a} (H_{\mathcal{A}}), K \in K(H) ,$ such that $ DFC=I+K     .$ By Theorem 2.3, we have then that $D \in \mathcal{M}\Phi_{-}(H_{\mathcal{A}})      .$ So $D \in \mathcal{M}\Phi_{+}(H_{\mathcal{A}}) \cap \mathcal{M}\Phi_{-}(H_{\mathcal{A}})      .$ But, $ \mathcal{M}\Phi(H_{\mathcal{A}}) = \mathcal{M}\Phi_{+}(H_{\mathcal{A}}) \cap \mathcal{M}\Phi_{-}(H_{\mathcal{A}})   $ by Corollary 2.4, so $ D \in \mathcal{M}\Phi(H_{\mathcal{A}})     .$ The proof of the second statement of Corollary 2.8 is similar.
\end{proof}
\begin{corollary}
If $D \in \mathcal{M}\Phi (H_{\mathcal{A}})$ and $ DF \in \mathcal{M}\Phi (H_{\mathcal{A}})$, then $F \in \mathcal{M}\Phi (H_{\mathcal{A}})$. If $ F \in \mathcal{M}\Phi (H_{\mathcal{A}})$ and $DF \in \mathcal{M}\Phi (H_{\mathcal{A}}),$ then $D \in \mathcal{M}\Phi (H_{\mathcal{A}}) .$
\end{corollary}
\begin{proof}
Suppose that $D \in \mathcal{M}\Phi (H_{\mathcal{A}})$ and $ DF \in \mathcal{M}\Phi (H_{\mathcal{A}}).$ Since $  DF \in \mathcal{M}\Phi (H_{\mathcal{A}}),$ by the proof of \cite[Lemma 2.7.15] {MT}  there exist some $ C \in B^{a} (H_{\mathcal{A}}), K \in K(H) $ s.t $	DFC=I+K.$\\
Moreover, since $ D \in \mathcal{M}\Phi(H_{\mathcal{A}})     ,$ by the proof of \cite[Lemma 2.7.15] {MT} there exist some $ D^{\prime} \in B^{a} (H_{\mathcal{A}}), K^{\prime} \in K(H)$ s.t. $D^{\prime}D=I+K^{\prime}.$ Hence $D^{\prime}DFC=D^{\prime}(DFC)=D^{\prime}(I+K)=D^{\prime}+D^{\prime}K$ and $D^{\prime}DFC=(D^{\prime}D)FC=(I+K^{\prime})FC=FC+K^{\prime}FC.$ Thus $D^{\prime}+D^{\prime}K=FC+K^{\prime}FC.$ Hence $ D^{\prime}D+D^{\prime}KD=FCD+K^{\prime}FCD.$ But $D^{\prime}D=I+K^{\prime}      ,$ so we obtain $ I+ K^{\prime} +D^{\prime}KD=FCD+K^{\prime}FCD.$ So $ FCD=I+ K^{\prime} +  D^{\prime}KD - K^{\prime}FCD.$
Since $(K^{\prime} +  D^{\prime}KD - K^{\prime}FCD) \in K(H_{\mathcal{A}})     ,$ by Theorem 2.3 we have that $  F \in \mathcal{M}\Phi_{-}(H_{\mathcal{A}})     .$ Now, since $ DF \in \mathcal{M}\Phi(H_{\mathcal{A}}) \subseteq \mathcal{M}\Phi_{+}(H_{\mathcal{A}})     ,$ by Corollary 2.6 it follows that $F \in \mathcal{M}\Phi_{+}(H_{\mathcal{A}})      $ also. Hence $$F \in \mathcal{M}\Phi_{+}(H_{\mathcal{A}}) \cap \mathcal{M}\Phi_{-}(H_{\mathcal{A}})=\mathcal{M}\Phi(H_{\mathcal{A}})      $$ by Corollary 2.5. The proof of the second statement of Corollary 2.9 is similar. 
\end{proof}
\begin{corollary}
$ \mathcal{M}\Phi_{+}(H_{\mathcal{A}}) \setminus \mathcal{M}\Phi(H_{\mathcal{A}}) $ and $\mathcal{M}\Phi_{-}(H_{\mathcal{A}}) \setminus \mathcal{M}\Phi(H_{\mathcal{A}}) $ are two sided ideals in $\mathcal{M}\Phi_{+}(H_{\mathcal{A}}) $ and $\mathcal{M}\Phi_{-}(H_{\mathcal{A}}) $ respectively. In particular, they are semigroups under multiplication.
\end{corollary}
\begin{proof}
Let $F,D \in \mathcal{M}\Phi_{+}(H_{\mathcal{A}}) $ and suppose first that $F \in \mathcal{M}\Phi_{+}(H_{\mathcal{A}}) \setminus \mathcal{M}\Phi(H_{\mathcal{A}}) .$\\
Since $\mathcal{M}\Phi_{+}(H_{\mathcal{A}}) $ is a semigroup by  Corollary 2.5 $DF \in \mathcal{M}\Phi_{+}(H_{\mathcal{A}}) .$ Now, if $DF \in \mathcal{M}\Phi(H_{\mathcal{A}}) ,$ by  Corollary 2.8 we have $ D \in \mathcal{M}\Phi(H_{\mathcal{A}}) .$ Then, by  Corollary 2.9, it would follow that $F \in \mathcal{M}\Phi(H_{\mathcal{A}}) ,$ which is a contradiction. Thus we must have that $DF \in \mathcal{M}\Phi_{+}(H_{\mathcal{A}}) \setminus \mathcal{M}\Phi(H_{\mathcal{A}})  .$ Suppose next that $D \in \mathcal{M}\Phi_{+}(H_{\mathcal{A}}) \setminus \mathcal{M}\Phi(H_{\mathcal{A}})  .$ Again, if $DF \in \mathcal{M}\Phi(H_{\mathcal{A}}) ,$ then, since $D \in \mathcal{M}\Phi_{+}(H_{\mathcal{A}}) ,$ by  Corollary 2.8 we would have that $D \in \mathcal{M}\Phi(H_{\mathcal{A}}) ,$ which is impossible. So, also in this case, we must have that $DF \in \mathcal{M}\Phi_{+}(H_{\mathcal{A}}) \setminus \mathcal{M}\Phi(H_{\mathcal{A}}) .$\\
Similarly one can prove the statement for $\mathcal{M}\Phi_{-}(H_{\mathcal{A}}) \setminus \mathcal{M}\Phi(H_{\mathcal{A}}) $
\end{proof}
\begin{corollary}
Let $ F \in B^{a}(M,N) .$ Then $ F \in \mathcal{M}\Phi_{+}(M,N) $ if and only if $ F^{*} \in \mathcal{M}\Phi_{-}(N,M) .$ Moreover, if $F \in \mathcal{M}\Phi(H_{\mathcal{A}})  ,$ then $ F^{*} \in \mathcal{M}\Phi(H_{\mathcal{A}}) $ and $ {\rm index} F = - {\rm index} F^{*}.$
\end{corollary}
\begin{proof}
Observe that it follows from the proofs of Theorems 2.2 and 2.3 part $1) \Rightarrow 2) $ which could be generalized to the case when $F \in B^{a}(M,N) $ (and not only when $F \in B^{a} (H_{\mathcal{A}}))$ that if $F \in \mathcal{M}\Phi_{+}(M,N) $ then for $F$ and consequently for $F^{*} $ there exist decompositions
\begin{center}
	$M = M_{1} \oplus {M_{1}^{\bot}}‎‎\stackrel{F}{\longrightarrow} M_{2} \oplus M_{2}^{\bot}= N $
\end{center}
\begin{center}
	$N = M_{2} \oplus {M_{2}^{\bot}}‎‎\stackrel{F^{*}}{\longrightarrow} M_{1} \oplus M_{1}^{\bot}= M $\\
\end{center}
with respect to which $F$ and $F^{*}$ have matrices 
\begin{center}
	$\left\lbrack
	\begin{array}{ll}
	F_{1} & F_{2} \\
	0 & F_{4} \\
	\end{array}
	\right \rbrack
	$,
\end{center} 
\begin{center}
	$\left\lbrack
	\begin{array}{ll}
	F_{1}^{*} & 0 \\
	F_{2}^{*} & F_{4}^{*} \\
	\end{array}
	\right \rbrack,
	$
\end{center} 
respectively, where $F_{1},F_{1}^{*}$ are isomorphisms and $M_{1}^{\bot}$ is finitely generated. Using the technique of diagonalization as in the proof of \cite[ Lemma 2.7.10] {MT}, we deduce that $ F^{*} \in \mathcal{M}\Phi_{-}(N,M)$ as $M_{1}^{\perp}$ is finitely generated. The proof is analogue when $ F \in \mathcal{M}\Phi_{-}(N,M),$ only in this case $M_{2}^{\perp}$ is finitely generated. If in addition $F$ is in $\mathcal{M}\Phi(H_{\mathcal{A}}) $, then both $M_{1}^{\perp}$ and $M_{2}^{\bot}$ will be finitely generated. Using again the tecnique of diagonalization, one deduces easily that $F^{*} \in \mathcal{M}\Phi(H_{\mathcal{A}}) $ in this case and ${\rm index} F=[M_{1}^{\bot}]-[M_{2}^{\bot}]$, ${\rm index} F^{*}=[M_{2}^{\bot}]-[M_{1}^{\bot}]$, so ${\rm index} F=- {\rm index}  F^{*} .$
\end{proof}
\begin{lemma}
Let M be a closed submodule of $H_{\mathcal{A}}$ s.t. $H_{\mathcal{A}} = M \tilde \oplus N$ for some finitely generated submodule $N.$
Let $ F \in B^{a}(H_{\mathcal{A}})$ , $J_{M}$ be the inclusion map from $M$ into $H_{\mathcal{A}}$ and suppose that $FJ_{M} \in \mathcal{M}\Phi_{+}(M,H_{\mathcal{A}}).$
Then $F \in \mathcal{M}\Phi_{+}(H_{\mathcal{A}}).$
\end{lemma}
\begin{proof}
Consider a decomposition $M= M_{1} \tilde \oplus {M_{2}}‎‎\stackrel{FJ_{M}}{\longrightarrow} \tilde M_{1} \tilde \oplus \tilde M_{2}= H_{\mathcal{A}} $ with respect to which 
\begin{center}
	$FJ_{M}=\left\lbrack
	\begin{array}{ll}
	(FJ_{M})_{1} & 0 \\
	0 & (FJ_{M})_{4} \\
	\end{array}
	\right \rbrack,
	$
\end{center}
where $(FJ_{M})_{1}$ is an isomorphism and $M_{2}$ is finitely generated. Then $F$ has the matrix 
\begin{center}
	$\left\lbrack
	\begin{array}{ll}
	F_{1} & F_{2} \\
	0 & F_{4} \\
	\end{array}
	\right \rbrack
	$
\end{center}
with respect to the decomposition 
$$H_{\mathcal{A}}= M_{1} \tilde \oplus (M_{2} \tilde \oplus {N)}‎‎\stackrel{F}{\longrightarrow} \tilde M_{1} \tilde \oplus \tilde M_{2}= H_{\mathcal{A}},$$
where $F_{1}$ is an isomorphism. Using the technique of diagonalization as in the proof of \cite[ Lemma 2.7.10] {MT} and the fact that $ M_{2} \tilde \oplus N$ is finitely generated as both $M_{2}$ and $N$ are so, we deduce that $ F \in {\mathcal{M}\Phi_{+}}(H_{\mathcal{A}}) . $
\end{proof}
Suppose now that $F \in B^{a}(H_{\mathcal{A}})$ and that ${\rm ran}F$ is closed. Then, again by \cite[ Theorem 2.3.3 ] {MT}, ${\rm ran}F$ is orthogonally complementable in $H_{\mathcal{A}}$, so $J_{{\rm ran}F} \in B^{a}({\rm ran}F,H_{\mathcal{A}})$.
\begin{lemma}
Suppose that $D,F \in B^{a}(H_{\mathcal{A}})$ $DF \in \mathcal{M}\Phi_{+}(H_{\mathcal{A}}) $ and ${\rm ran}F$ is closed. Then
$DJ_{{\rm ran}F} \in \mathcal{M}\Phi_{+}({\rm ran}F,H_{\mathcal{A}}) .$
\end{lemma}
\begin{proof}
Let $H_{\mathcal{A}} = M_{1} \tilde \oplus {N_{1}}‎‎\stackrel{DF}{\longrightarrow} M_{2} \tilde \oplus N_{2}= H_{\mathcal{A}} $ be the decomposition with respect to which $DF$ has the matrix
\begin{center}
	$	\left\lbrack
	\begin{array}{ll}
	(DF)_{1} & 0 \\
	0 & (DF)_{4} \\
	\end{array}
	\right \rbrack,
	$
\end{center} 
where $(DF)_{1}$ is an isomorphism and $N_{1}$ is finitely generated.\\
Since $D(D^{-1}(M_{2}) \cap {\rm ran}F)=M_{2}$, that is
$DJ_{{\rm ran}F}(D^{-1}(M_{2})\cap {\rm ran}F)=M_{2}$ , we get that
$${\rm ran}F=((D^{-1}(M_{2}) \cap {\rm ran}F) \tilde \oplus ((D^{-1}(N_{2}) \cap {\rm ran}F).$$
To see this, let $y \in {\rm ran}F$, then $Dy=Dx_{m}+x_{n}$, for some $ x_{m} \in D^{-1}(M_{2}) \cap {\rm ran}F$ and for some $x_{n} \in N_{2}$.\\
Hence $x_{n} \in D(y-x_{m})$, so $x_{n} \in D({\rm ran}F)$. As
$x_{n} \in N_{2}$ also, we have $x_{n} \in D({\rm ran}F) \cap N_{2}$.
Thus $(y-x_{m}) \in {\rm ran}F \cap D^{-1}(N_{2})$ . Since $y=x_{m}+(y-x_{m}), x_{m} \in D^{-1}(M_{2}) \cap {\rm ran}F$
and
$(y-x_{m}) \in D^{-1}(N_{2}) \cap {\rm ran}F,$ we deduce that 
$${\rm ran}F=(D^{-1}(M_{2}) \cap {\rm ran}F) \tilde \oplus ((D^{-1}(N_{2}) \cap {\rm ran}F).$$
as $ y \in {\rm ran}F$ was arbitrary. With respect to the decomposition 
$${\rm ran}F=(D^{-1}(M_{2}) \cap {\rm ran}F) \tilde \oplus ((D^{-1}(N_{2}) \cap {\rm ran}F)‎‎\stackrel{DJ_{{\rm ran}F}}{\longrightarrow} M_{2} \tilde \oplus N_{2}= H_{\mathcal{A}} $$ 
$DJ_{{\rm ran}F}$ has the matrix 
\begin{center}
	$\left\lbrack
	\begin{array}{ll}
	(DJ_{{\rm ran}F})_{1} & 0 \\
	0 & (DJ_{{\rm ran}F})_{4} \\
	\end{array}
	\right \rbrack,
	$
\end{center}
where $(DJ_{{\rm ran}F})_{1}$ is an isomorphism.
Now, since $DF$ has the matrix 
\begin{center}
	$	\left\lbrack
	\begin{array}{ll}
	(DF)_{1} & 0 \\
	0 & (DF)_{4} \\
	\end{array}
	\right \rbrack
	$
\end{center}
with respect to the decomposition $$H_{\mathcal{A}} = M_{1} \tilde \oplus {N_{1}} ‎‎\stackrel{DF}{\longrightarrow} M_{2} \tilde \oplus N_{2}= H_{\mathcal{A}} $$ it is easily seen that $ D^{-1}(N_{2}) \cap {\rm ran}F=F(N_{1})$ which is finitely generated. We are done.
\end{proof}
\begin{corollary}
Let V be a finitely generated Hilbert submodule of $H_{\mathcal{A}},$ $F \in B^{a}(H_{\mathcal{A}})$ and suppose that 
$P_{V^{\bot}} F \in \mathcal{M}\Phi_{-} (H_{\mathcal{A}},V^{\bot})$, where $P_{V^{\bot}}$ is the orthogonal projection onto $V^{\bot}$ along $V.$ Then $F \in \mathcal{M}\Phi_{-}(H_{\mathcal{A}}) .$
\end{corollary}
\begin{proof}
Passing to the adjoints and using  Lemma 2.12 together with  Corollary 2.11, one obtains the result. 
\end{proof}
\begin{corollary}
Let $D, F \in B^{a}(H_{\mathcal{A}})$ and suppose that ${\rm ran}D^{*} $ is closed. If $DF \in \mathcal{M}\Phi_{-} (H_{\mathcal{A}})$ , then
$P_{ \ker (D)^{\bot}}F \in \mathcal{M}\Phi_{-} (H_{\mathcal{A}},{\rm ran}(D^{*}))$
\end{corollary}
\begin{proof}
Observe that since ${\rm ran}(D^{*})$ is closed, then by the proof of \cite[ Theorem 2.3.3] {MT}, we have that $H_{\mathcal{A}}= \ker (D) \oplus {\rm ran}(D^{*}) .$ Hence $ \ker (D)^{\perp}={\rm ran}(D^{*}) .$ Passing to the adjoints and using  Lemma 2.13 together with  Corollary 2.11, one deduces the corollary.
\end{proof}
\begin{lemma}
Let $F \in \mathcal{M}\Phi(H_{\mathcal{A}}) $ and suppose that there are two decompositions
$$H_{\mathcal{A}} = M_{1} \tilde \oplus {N_{1}}‎‎\stackrel{F}{\longrightarrow} M_{2} \tilde \oplus N_{2}= H_{\mathcal{A}} $$
$$H_{\mathcal{A}} = M_{1}^{\prime} \tilde \oplus {N_{1}^{\prime}}‎‎\stackrel{F}{\longrightarrow} M_{2}^{\prime} \tilde \oplus N_{2}^{\prime}= H_{\mathcal{A}} $$
with respect to which F has matrices
\begin{center}
	$\left\lbrack
	\begin{array}{ll}
	F_{1} & 0 \\
	0 & F_{4} \\
	\end{array}
	\right \rbrack
	, $
	$\left\lbrack
	\begin{array}{ll}
	F_{1}^{\prime} & 0 \\
	0 & F_{4}^{\prime} \\
	\end{array}
	\right \rbrack,
	 $
\end{center}
respectively, where $F_{1},{F_{1}}^{\prime}$ are isomorphisms, $N_{1},{N_{1}}^{\prime},N_{2}$ are closed, finitely generated and ${N_{2}}^{\prime}$ is just closed. Then ${N_{2}}^{\prime}$ is finitely generated also.
\end{lemma}
\begin{proof} Since $N_{1},{N_{1}}^{\prime}$ are finitely generated, by \cite[ Theorem 2.7.5] {MT}, there exist an $n$ such that\\
\begin{center}
	$L_{n}=P \tilde \oplus p_{n}(N_{1}) , P=M_{1}\cap L_{n} , p_{n}(N_{1}) \cong N_{1} $ and\\
\end{center}
$L_{n}=P^{\prime} \tilde \oplus p_{n}(N_{1}^{\prime})$ , $P^{\prime}=M_{1}^{\prime}\cap L_{n}$ , $p_{n}(N_{1}^{\prime}) \cong N_{1}^{\prime} $\\
Then $$H_{\mathcal{A}} = L_{n}^{\bot} \tilde \oplus P \tilde \oplus N_{1}=L_{n}^{\bot} \tilde \oplus P^{\prime} \tilde \oplus N_{1}^{\prime} ,$$\\
consequently $\sqcap_{{M_{1}}_{{\mid}_{(L_{n}^{\bot} \tilde \oplus P)}}} $ , 
$\sqcap_{{M_{1}}_{{\mid}_{(L_{n}^{\bot} \tilde \oplus P^{\prime})}}^{\prime}} $ are isomorphisms from $L_{n}^{\bot} \tilde \oplus P$ onto $M_{1}$ and from $L_{n}^{\bot} \tilde \oplus P^{\prime}$ onto $M_{1}^{\prime}$ respectively, where $\sqcap_{{M_{1}}_{{\mid}_{(L_{n}^{\bot} \tilde \oplus P)}}} $ , $\sqcap_{{M_{1}}_{{\mid}_{(L_{n}^{\bot} \tilde \oplus P^{\prime})}}^{\prime}} $
denote  the restrictions of projections onto $M_{1}$ and $M_{1}^{\prime}$ along $N_{1}$ and $N_{1}^{\prime}$ restricted to  $L_{n}^{\perp} \tilde \oplus P $  and $L_{n}^{\bot} \tilde \oplus P^{\prime} $ respectively. Since $F(M_{1})=M_{2}$ and $F(N_{1})\in N_{2}$ and $H_{\mathcal{A}}= M_{1} \tilde{\oplus} N_{1}$ it follows that 
$$\sqcap_{M_{2}} F_{{\mid}_{(L_{n}^{\bot} \tilde \oplus P)}}=F \sqcap_{{M_{1}}_{{\mid}_{(L_{n}^{\bot} \tilde \oplus P)}}}=F_{1} \sqcap_{{M_{1}}_{{\mid}_{(L_{n}^{\bot} \tilde \oplus P^{\prime})}}^{\prime}} $$ where $\sqcap_{M_{2}}$ denotes the projection onto $M_{2}$ along $N_{2}$. Hence $\sqcap_{M_{2}} F_{{\mid}_{(L_{n}^{\bot} \tilde \oplus P)}}$ is an isomorphism as $\sqcap_{{M_{1}}_{{\mid}_(L_{n}^{\bot} \tilde \oplus P)}}$ and $F_{1}$ are so. Similarly, $\sqcap_{M_{2}^{\prime}} F_{{\mid}_{(L_{n}^{\bot} \tilde \oplus P^{\prime})}}$ is an isomorphism, where $\sqcap_{M_{2}^{\prime}}$ denotes the projection onto $M_{2}^{\prime}$ along $N_{2}^{\prime}$. \\
We get then that $F$ has the matrices
\begin{center}
	$\left\lbrack
	\begin{array}{ll}
	\tilde{F_{1}} & 0 \\
	\tilde F_{3} & F_{4} \\
	\end{array}
	\right \rbrack 
	$,
	$\left\lbrack
	\begin{array}{ll}
	\tilde{F_{1}^{\prime}} & 0 \\
	\tilde F_{3}^{\prime} & F_{4}^{\prime} \\
	\end{array}
	\right \rbrack 
	$ with respect to the decompositions
\end{center}
$$H_{\mathcal{A}} = (L_{n}^{\bot} \tilde \oplus P) \tilde \oplus {N_{1}}‎‎\stackrel{F}{\longrightarrow} M_{2} \tilde \oplus N_{2}= H_{\mathcal{A}} $$
$$H_{\mathcal{A}} = (L_{n}^{\bot} \tilde \oplus P^{\prime}) \tilde \oplus {N_{1}^{\prime}}‎‎\stackrel{F}{\longrightarrow} M_{2}^{\prime} \tilde \oplus N_{2}= H_{\mathcal{A}},$$
respectively, where $\tilde{F_{1}}= \sqcap_{M_{2}} F_{{\mid}_{(L_{n}^{\bot} \tilde \oplus P)}}$ , $\tilde {F_{1}^{\prime}}=\sqcap_{M_{2}^{\prime}} F_{{\mid}_{(L_{n}^{\bot} \tilde \oplus P^{\prime})}}$ are isomorphisms. As in the proof of \cite[ Lemma 2.7.11] {MT}, we let
\begin{center}
	$V=	\left\lbrack
	\begin{array}{ll}
	1 & 0 \\
	- \tilde{F_{3}}\tilde{F_{1}}^{-1} & 1 \\
	\end{array}
	\right \rbrack
	,$
	$V^{\prime}=	\left\lbrack
	\begin{array}{ll}
	1 & 0 \\
	- \tilde{F^{\prime}_{3}}\tilde{F^{\prime}_{1}}^{-1} & 1 \\
	\end{array}
	\right \rbrack
	,$
\end{center}
with respect to the decomposition 
$$H_{\mathcal{A}}= M_{2} \tilde \oplus {N_{2}}‎‎\stackrel{V}{\longrightarrow} M_{2} \tilde \oplus N_{2}= H_{\mathcal{A}} $$
$$H_{\mathcal{A}}= M_{2}^{\prime} \tilde \oplus {N_{2}^{\prime}}‎‎\stackrel{V}{\longrightarrow} M_{2}^{\prime} \tilde \oplus N_{2}^{\prime}= H_{\mathcal{A}}. $$

Then $F$ has the matrices 
\begin{center}
	$	\left\lbrack
	\begin{array}{ll}
	\tilde {\tilde {F_{1}}} & 0 \\
	0 & \tilde {\tilde {F_{4}}} \\
	\end{array}
	\right \rbrack
	$,
	$	\left\lbrack
	\begin{array}{ll}
	\tilde {\tilde {F_{1}}}^{\prime} & 0 \\
	0 & \tilde {\tilde {F_{4}}}^{\prime} \\
	\end{array}
	\right \rbrack
	$
\end{center}
with respect to the decompositions
$$H_{\mathcal{A}}=(L_{n}^{\bot} \tilde \oplus P)\tilde \oplus {N_{1}}‎‎\stackrel{F}{\longrightarrow} V^{-1}(M_{2}) \tilde \oplus V^{-1}(N_{2})=H_{\mathcal{A}}$$
$$H_{\mathcal{A}}=(L_{n}^{\bot} \tilde \oplus P^{\prime})\tilde \oplus {N_{1}^{\prime}}‎‎\stackrel{F}{\longrightarrow} {V^{\prime}}^{-1}(M_{2}^{\prime}) \tilde \oplus {V^{\prime}}^{-1}(N_{2}^{\prime})=H_{\mathcal{A}},$$ 
respectively, where $\tilde {\tilde {F_{1}}}, \tilde {\tilde {F_{1}}}^{\prime}$ are isomorphism. Again, as in the proof of \cite[ Lemma 2.7.11] {MT}, we change these decompositions into
$$H_{\mathcal{A}}=L_{n}^{\bot} \tilde \oplus (P\tilde \oplus N_{1})‎‎\stackrel{F}{\longrightarrow} F(L_{n}^{\bot}) \tilde \oplus(F(P)\tilde \oplus V^{-1}(N_{2}))=H_{\mathcal{A}}$$
$$H_{\mathcal{A}}=L_{n}^{\bot} \tilde \oplus (P^{\prime}\tilde \oplus N_{1}^{\prime})‎‎\stackrel{F}{\longrightarrow} F(L_{n}^{\bot}) \tilde \oplus(F(P^{\prime})\tilde \oplus {V^{\prime}}^{-1}(N_{2}^{\prime}))=H_{\mathcal{A}}$$
and with respect to these decompositions $F$ has matrices
\begin{center}
	$	\left\lbrack
	\begin{array}{ll}
	\tilde {\tilde {F_{1}}} & 0 \\
	0 & \tilde {\tilde {F_{4}}} \\
	\end{array}
	\right \rbrack
	$,
	$	\left\lbrack
	\begin{array}{ll}
	\tilde {\tilde {F_{1}}}^{\prime} & 0 \\
	0 & \tilde {\tilde {F_{4}}}^{\prime} \\
	\end{array}
	\right \rbrack,
	$ respectively, where $\tilde {\tilde {F_{1}}} , \tilde {\tilde {F_{1}}}^{\prime}$ are isomorphisms. 
\end{center}
As $$H_{\mathcal{A}}=F(L_{n}^{\bot}) \tilde \oplus (F(P)\tilde \oplus V^{-1}(N_{2})) = F(L_{n}^{\bot}) \tilde \oplus(F(P^{\prime})\tilde \oplus {V^{\prime}}^{-1}(N_{2}^{\prime}))=H_{\mathcal{A}}$$
clearly we have
$$(F(P)\tilde \oplus V^{-1}(N_{2}))\cong (F(P^{\prime})\tilde \oplus {V^{\prime}}^{-1}(N_{2}^{\prime})).$$ 
Now, $F(P)$ and $ V^{-1}(N_{2})$ are finitely generated since $F_{{\mid}_P}$ , $V^{-1}$ are isomorphisms and $P,N_{2}$ are finitely generated. Hence $(F(P)\tilde \oplus V^{-1}(N_{2}))$ is finitely generated, consequently $(F(P^{\prime})\tilde \oplus {V^{\prime}}^{-1}(N_{2}^{\prime}))$ 
is finitely generated being isomorphic to a finitely generated submodule $(F(P)\tilde \oplus V^{-1}(N_{2}))$. Therefore, ${V^{\prime}}^{-1}(N_{2}^{\prime})$ is finitely generated, since it is generated by the images of the generators of $F(P^{\prime})\tilde \oplus {V^{\prime}}^{-1}(N_{2}^{\prime})$ under the projection onto ${V^{\prime}}^{-1}(N_{2}^{\prime})$ along $F(P^{\prime})$. But $V^{\prime}$ is an isomorphism, hence $N_{2}^{\prime}$ must be finitely generated.
\end{proof}
\begin{lemma}
Let $F \in \mathcal{M}\Phi (H_{\mathcal{A}})$ and let $$H_{\mathcal{A}} = M_{1} \tilde \oplus {N_{1}}‎‎\stackrel{F}{\longrightarrow} M_{2} \tilde \oplus N_{2}= H_{\mathcal{A}} $$ 
be a decomposition with respect to which F has the matrix
\begin{center}
	$\left\lbrack
	\begin{array}{ll}
	F_{1} & 0 \\
	0 & F_{4} \\
	\end{array}
	\right \rbrack,
	$
\end{center}
where $F_{1}$ is an isomorphism, $N_{2}$ is finitely generated and $N_{1}$ is just closed. Then $N_{1}$ is finitely generated.
\end{lemma}
\begin{proof}
By  Corollary 2.11, we have $F^{*} \in \mathcal{M}\Phi(H_{\mathcal{A}})$ since $F \in \mathcal{M}\Phi (H_{\mathcal{A}})$ and moreover by the proof of Theorem 2.3, $N_{1},N_{2}$ are orthogonally complementable. With respect to the decomposition 
$$H_{\mathcal{A}} = {N_{1}}^{\bot} \tilde \oplus {N_{1}}‎‎\stackrel{F}{\longrightarrow} {N_{2}}^{\bot} \tilde \oplus N_{2}=H_{\mathcal{A}}$$ 
$F$ has the matrix
\begin{center}
	$	\left\lbrack
	\begin{array}{ll}
	{\tilde F_{1}} & 0 \\
	{\tilde F_{3}} & { F_{4}} \\
	\end{array}
	\right \rbrack,
	$
\end{center}
where $\tilde F_{1}$ is an isomorphism and hence with respect to the decomposition $$H_{\mathcal{A}} = {N_{2}}^{\bot} \tilde \oplus {N_{2}}‎‎\stackrel{F^{*}}{\longrightarrow} {N_{1}}^{\bot} \tilde \oplus N_{1}=H_{\mathcal{A}}$$
$F^{*}$ has the matrix
\begin{center}
	$\left\lbrack
	\begin{array}{ll}
	\tilde F_{1}^{*} &\tilde F_{3}^{*} \\
	0 &\tilde F_{4}^{*} \\
	\end{array}
	\right \rbrack .
	$
\end{center} 
Clearly, $\tilde F_{1}^{*}$ is an isomorphism as $\tilde F_{1}$ is so. Then
$F^{*}$ has the matrix
\begin{center}
	$\left\lbrack
	\begin{array}{ll}
	\tilde F_{1}^{*} &0 \\
	0 & F_{4}^{*} \\
	\end{array}
	\right \rbrack
	$ with respect to the decomposition
\end{center}
$$H_{\mathcal{A}} = U({N_{2}}^{\bot}) \tilde \oplus U( N_{2})‎‎\stackrel{F^{*}}{\longrightarrow} {N_{1}}^{\bot} \tilde \oplus N_{1}=H_{\mathcal{A}},$$
where $U$ is an isomorphism. But since $F^{*} \in \mathcal{M}\Phi(H_{\mathcal{A}})$, $\tilde F_{1}^{*}$ is an isomorphism and $U({N_{2}})$ is finitely generated (as $N_{2}$ is finitely generated by assumption ), we can use the previous lemma to deduce that $N_{1}$ is finitely generated. \end{proof}
\begin{corollary}
Let $ F \in \mathcal{M}\Phi_{+}(H_{\mathcal{A}}) $ and let 
$$ H_{\mathcal{A}} = M_{1} \tilde \oplus {N_{1}}‎‎\stackrel{F}{\longrightarrow} M_{2} \tilde \oplus N_{2}= H_{\mathcal{A}} $$ 
$$ H_{\mathcal{A}} = \tilde{M_{1}} \tilde \oplus {\tilde{N_{1}}}‎‎\stackrel{F}{\longrightarrow} \tilde M_{2} \tilde \oplus \tilde N_{2}= H_{\mathcal{A}} $$ 
be two $ \mathcal{M}\Phi_{+} $ decompositions for $F$. Then there exists some finitely generated submodules $ P $ and $ \tilde P .$ s.t. $(N_{2} \tilde \oplus P) \cong (\tilde{N_{2}} \tilde \oplus \tilde P). $
\end{corollary}
\begin{proof}
Statement follows from the proof of  Lemma 2.16.
\end{proof}
\begin{corollary}
Let $ D \in \mathcal{M}\Phi_{-}(H_{\mathcal{A}}) $ and let $$H_{\mathcal{A}} = M_{1}^{\prime} \tilde \oplus {N_{1}^{\prime}}‎‎\stackrel{D}{\longrightarrow} M_{2}^{\prime} \tilde \oplus N_{2}^{\prime}= H_{\mathcal{A}} $$
$$H_{\mathcal{A}} = \tilde M_{1}^{\prime} \tilde \oplus {\tilde{N_{1}}}‎‎\stackrel{D}{\longrightarrow} \tilde M_{2}^{\prime} \tilde \oplus \tilde N_{2}^{\prime}= H_{\mathcal{A}} $$
be two $ \mathcal{M}\Phi_{-} $ decompositions for $D.$ Then there exists some finitely generated,closed submodules $P^{\prime} $ and $ \tilde{P^{\prime}}$ s.t. $(N_{1}^{\prime} \tilde \oplus P^{\prime}) \cong (\tilde{N_{1}}^{\prime} \tilde \oplus \tilde P^{\prime}) .$
\end{corollary}
\begin{proof}	
Statement follows from the proof of  Lemma 2.17.
\end{proof}
\begin{lemma}
Let $F \in \mathcal{M}\Phi_{+}(H_{\mathcal{A}}) $ and suppose that ${\rm ran} F$ is closed.
If 
$$H_{\mathcal{A}} = M_{1} \tilde \oplus {N_{1}}‎‎\stackrel{F}{\longrightarrow} M_{2} \tilde \oplus N_{2}= H_{\mathcal{A}} $$ 
$$H_{\mathcal{A}} = M_{1}^{\prime} \tilde \oplus {N_{1}^{\prime}}‎‎\stackrel{F}{\longrightarrow} M_{2}^{\prime} \tilde \oplus N_{2}^{\prime}= H_{\mathcal{A}} $$
are two $\mathcal{M}\Phi_{+} $ decomposition for $F$ then $F(N_{1}),F(N_{1}^{\prime})$ are closed finitely generated projective modules and $$[N_{1}]-[F(N_{1})]=[N_{1}^{\prime}]-[F(N_{1}^{\prime})] $$ in $K(A).$
\end{lemma}
\begin{proof}
First of all, it is obvious that $$F(N_{1})={\rm ran} F \cap N_{2}, F(N_{1}^{\prime})={\rm ran} F\cap N_{2}^{\prime}.$$ 
Now, since ${\rm ran} F$ is closed (by assumption), we have that $F(N_{1}),F(N_{1}^{\prime})$ are closed and also finitely generated as $N_{1}, N_{1}^{\prime} $ are so. Then, by \cite[ Lemma 2.3.7] {MT} 
$$F(N_{1}) \tilde \oplus \tilde N_{2}= N_{2},F(N_{1}^{\prime}) \tilde \oplus \tilde N_{1}^{\prime}= N_{2}^{\prime} .$$ for some closed submodules $\tilde N_{1}^{\prime},\tilde N_{2}^{\prime}$ of $N_{1},N_{2}$ respectively.\\
Thus $H_{\mathcal{A}}=M_{2}\tilde \oplus \tilde N_{2} \tilde \oplus F(N_{1})$ and $H_{\mathcal{A}}=M_{2}^{\prime} \tilde \oplus \tilde N_{2}^{\prime} \tilde \oplus F(N_{1}^{\prime})$. Since $F(N_{1}),F(N_{1}^{\prime})$ are finitely generated, from \cite[ Theorem 2.7.5] {MT} it follows that $F(N_{1}),F(N_{1}^{\prime})$ are projective. Moreover, again by \cite[ Theorem 2.7.5] {MT}, we may assume that there exists some $m$ such that 
$$N_{1}\subseteq L_m,L_m=N_{1} \tilde \oplus \tilde P_{1},M_{1}= \tilde P_{1} \tilde \oplus L_m$$ and 
$$ L_m=\tilde P_{1}^{\prime} \tilde \oplus p_m(N_{1}^{\prime}),p_m(N_{1}^{\prime}) \tilde = N_{1}^{\prime},$$
where $p_m$ is the projection onto $L_m$ along $L_m^{\bot}$ and $P_{1}^{\prime}, \tilde{P_{1}}$ are projective, finitely generated A-modules.\\
Set $L_m^{\prime}=F(L_m)+F(N_{1}) $ and ${L_m^{\prime \prime}}=F(L_m^{\bot})$. Note that ${\rm ran} F=L_m^{\prime} \tilde \oplus {L_m^{\prime \prime}}$. By the arguments similar to the proof of \cite[ Theorem 2.7.9] {MT}, we deduce that
$$[N_{1}]+ [\tilde P_{1}]=[N_{1}^{\prime}]+ [P_{1}^{\prime}]=[L_m] $$
$$[F(N_{1})]+[F(\tilde P_{1})]=[F(N_{1}^{\prime})]+[F(P_{1}^{\prime})]=[L_m^{\prime}], $$
$$[F(\tilde P_{1})]\cong[\tilde P_{1}], [F(P_{1}^{\prime})]\cong[P_{1}^{\prime}]$$
Hence $$[N_{1}]-[F(N_{1})]=[N_{1}^{\prime}]-[F(N_{1}^{\prime})]. $$
\end{proof}
\section{Generalized Schechter characterization of $\mathcal{M}\Phi_{+}$ operators on $H_{\mathcal{A}}$}
In this section we investigate the classes $\mathcal{M}\Phi_{+}(H_{\mathcal{A}})$, $B^{a}(H_{\mathcal{A}}) \setminus \mathcal{M}\Phi_{+}(H_{\mathcal{A}})$ and prove an analogue of some results concerning the classes $ \Phi_{+}(X)$ , $B(X) \setminus \Phi_{+}(X)$ (where $X$ is a Banach space) in \cite{S2}.\\
\begin{lemma}
Let $F \in B^{a}(M,N)$ Then $F \in \mathcal{M}\Phi_{+}(M,N) $ if and only if there exists a closed, orthogonally complementable submodule $M^{\prime} \subseteq M$ such that $F_{{\mid}_{M^{\prime}}}$ is bounded below and ${M^{\prime}}^\bot$ is finitely generated.
\end{lemma}
\begin{proof}
If such $M^{\prime}$ exists, then $F(M^{\prime})$ is closed in $N.$ Moreover, as $M^{\prime}$ is orthogonally complementable, $F_{{\mid}_{M^{\prime}}}$ is adjointable. By \cite[ Theorem 2.3.3] {MT}, $F(M^{\prime})$ is orthogonally complementable in $N$. Then with respect to the decomposition
$$M = M^{\prime} \oplus {{M^{\prime}}^{\bot}}‎‎\stackrel{F}{\longrightarrow} F(M^{\prime}) \oplus F( {M^{\prime}})^{\bot}=N ,$$
$F$ has the matrix
\begin{center}
	$\left\lbrack
	\begin{array}{ll}
	F_{1} & F_{2} \\
	0 & F_{4} \\
	\end{array}
	\right \rbrack,
	$
\end{center}
where $F_{1}$ is an isomorphism.Using the technique of diagonalization as in the proof of \cite[ Lemma 2.7.10] {MT} and the fact that $M^{{\prime}^{\perp}}$ is finitely generated, we deduce that $F \in \mathcal{M}\Phi_{+}(M,N) $. On the other hand if $F \in \mathcal{M}\Phi_{+}(M,N) $, by the similar arguments as in the proof of \cite[ Theorem 2.7.6 ] {MT} we may assume that there exists a decomposition
$$M = M^{\prime} \oplus {{M^{\prime}}^{\bot}}‎‎\stackrel{F}{\longrightarrow} N^{\prime} \oplus {N^{\prime \prime}}=N ,$$
with respect to which F has the matrix
\begin{center}
	$	\left\lbrack
	\begin{array}{ll}
	F_{1} & 0 \\
	0 & F_{4} \\
	\end{array}
	\right \rbrack,
	$
\end{center}
where $F_{1}$ is an isomorphism and ${M^{\prime}}^{\bot}$ is finitely generated. 
\end{proof}
\begin{lemma}
Let $F \in B^{a}(H_{\mathcal{A}}) \setminus \mathcal{M}\Phi_{+}(H_{\mathcal{A}}) $. Then there exists a sequence $\lbrace x_{k} \rbrace \subseteq H_{\mathcal{A}}$ and an increasing sequence $\lbrace n_{k} \rbrace \subseteq \mathbb{N}$ s.t.\\
$$x_{k} \in L_{{n}_{k}} \setminus L_{{n}_{k-1}} \text{ for all k } \in \mathbb{N},  \parallel x_{k} \parallel \leq 1 \text{ for all k } \in \mathbb{N} $$
and
$$\parallel Fx_{k} \parallel \leq 2^{1-2k} \text{ for all k } \in \mathbb{N} .$$ 
\end{lemma}
\begin{proof}   
Since $F \notin \mathcal{M}\Phi_{+}(H_{\mathcal{A}}) $, there exists an $$ \tilde x_{1} \subseteq H_{\mathcal{A}}, \parallel \tilde x_{1} \parallel \leq 1 \textmd{, s.t.} \parallel F \tilde x_{1} \parallel \leq \frac{1}{4}$$ because $F$ is then not bounded below by the previous lemma. As $$\parallel P_{L_{n_{1}}^\bot} \tilde x_{1} \parallel \longrightarrow 0 \textmd{ when } n \rightarrow \infty,$$ there exists an $ n_{1} \in \mathbb{N} $ such that $\parallel P_{{L_{n_{1}}^\bot}} \tilde x_{1} \parallel \leq \frac{1}{\parallel F \parallel} \frac{1}{4}$ (here again $P_{{L_{n_{1}}\bot}}$ denotes the orthogonal projection onto $L_{n}^{\bot} $ along $L_n$). Hence
$$\parallel FP_{{L_{n_{1}}}} \tilde x_{1} \parallel \leq \parallel F \tilde x_{1} \parallel + \parallel FP_{{L_{n_{1}}^\bot}} \tilde x_{1} \parallel \leq \frac{1}{4} + \frac{1}{4} =\frac{1}{2}.$$ 
Set $x_{1}=P_{{L_{n_{1}}}} \tilde x_{1} $, then $$\parallel x_{1} \parallel \leq \parallel \tilde x_{1} \parallel \leq 1, \parallel Fx_{1} \parallel \leq \frac{1}{2}$$ 
and $x_{1} \in L_{n_{1}} $. Suppose so that there exists $$x_{1},\dots,x_{k} \in H_{\mathcal{A}}, n_{1} \leq n_{2} \leq \dots \leq n_{k} $$ such that the hypothesis of the lemma holds. By previous lemma, $F$ is not bounded below on $L_{n_{k}}^{\bot},$ hence we can find an ${\tilde x_{k+1}} \in L_{n_{k}}^{\bot}$ such that $\parallel {\tilde x_{k+1}} \parallel =1 $ and $$\parallel F\tilde x_{k+1} \parallel \leq 2^{-2(k+1)} .$$
Again, since
$$\lim_{n \to 0} \parallel P_{{L_{n}^\bot}} \tilde x_{k+1} \parallel=0 ,$$ 
there exits an $n_{k+1} \geq n_{k} $ such that
$$\parallel P_{{L_{n_{k+1}}^\bot}} \tilde x_{k+1} \parallel \leq \frac{1}{\parallel F \parallel} 2^{-2(k+1)} .$$
Then
$$\parallel F P_{{L_{n_{k+1}}}} \tilde x_{k+1} \parallel \leq \parallel F \tilde x_{k+1} \parallel + \parallel FP_{{L_{n_{k+1}}^\bot}} \tilde x_{k+1} \parallel \leq 2^{-2(k+1)}+2^{-2(k+1)}= 2^{1-2(k+1)}.$$
Set $x_{k+1}=P_{{L_{n_{k+1}}}}\tilde x_{k+1} $. We then have $x_{k+1} \in L_{{n}_{k+1}} \setminus L_{{n}_{k}}$
(because $\tilde x_{k+1} \in {L_{n_{k}}^\bot}$),
$$\parallel x_{k+1} \parallel \leq \parallel \tilde x_{k+1} \parallel = 1 \mbox{ and } \parallel Fx_{k+1} \parallel \leq 2^{1-2(k+1)}.$$ 
By induction, the lemma follows. 
\end{proof}
\section{Openness of the set of semi-{$\mathcal{A}$}-Fredholm operators on $H_{\mathcal{A}}$}
In this section we prove that the sets $\mathcal{M}\Phi_{+}(H_{\mathcal{A}}) \setminus \mathcal{M}\Phi(H_{\mathcal{A}})$ and  $\mathcal{M}\Phi_{-}(H_{\mathcal{A}}) \setminus \mathcal{M}\Phi(H_{\mathcal{A}})$  are open in the norm topology, as an analogue of the result in \cite{S}.
Also, we derive some consequences. Recall that $\mathcal{M}\Phi(H_{\mathcal{A}})$  is open in the norm topology by \cite[ Lemma 2.7.10] {MT}.
\begin{theorem}
The sets $\mathcal{M}\Phi_{+}(H_{\mathcal{A}}) \setminus \mathcal{M}\Phi(H_{\mathcal{A}})$ and $\mathcal{M}\Phi_{-}(H_{\mathcal{A}}) \setminus \mathcal{M}\Phi(H_{\mathcal{A}})$ are open in $ B^{a}(H_{\mathcal{A}}),$ where $ B^{a}(H_{\mathcal{A}})$ is equipped with the norm topology.
\end{theorem}
\begin{proof}
Let $F \in \mathcal{M}\Phi_{+}(H_{\mathcal{A}}) \setminus \mathcal{M}\Phi(H_{\mathcal{A}}).$ Then there exists a decomposition $$H_{\mathcal{A}} = M_{1} \tilde \oplus {N_{1}}‎‎\stackrel{F}{\longrightarrow} M_{2} \tilde \oplus N_{2}= H_{\mathcal{A}} $$ with respect to which F has the matrix
\begin{center}
	$\left\lbrack
	\begin{array}{ll}
	F_{1} & 0 \\
	0 & F_{4} \\
	\end{array}
	\right \rbrack,
	$
\end{center}
where $F_{1}$ is an isomorphism, $N_{1}$ is closed finitely generated and $N_{2}$ is closed, but \underline{not} finitely generated. If $ D \in B^{a}(H_{\mathcal{A}})$ such that $\parallel D \parallel < \epsilon,$ then for $ \epsilon$ small enough we may by the same arguments as in the proof of \cite[ Lemma 2.7.10] {MT} find isomorphisms $U_{1}, U_{2}$ such that $F+D$ has the matrix
\begin{center}
	$\left\lbrack
	\begin{array}{ll}
	(F+D)_{1} & 0 \\
	0 & (F+D)_{4} \\
	\end{array}
	\right \rbrack
	$
\end{center}
with respect to the decomposition
$$H_{\mathcal{A}} = U_{1}(M_{1}) \tilde \oplus U_{1}(N_{1})‎‎\stackrel{F+D}{\longrightarrow} U_{2}^{-1}(M_{2}) \tilde \oplus U_{2}^{-1}(N_{2})= H_{\mathcal{A}},$$
where $(F+D)_{1}$ is an isomorphism. Since $U_{2}$ is an isomorphism and $N_{2}$ is not finitely generated, it follows that $U_{2}^{-1}(N_{2})$ is not finitely generated. Now, as $F+D$ has the matrix
\begin{center}
	$\left\lbrack
	\begin{array}{ll}
	(F+D)_{1} & 0 \\
	0 & (F+D)_{4} \\
	\end{array}
	\right \rbrack
	$
\end{center} 
with respect to the decomposition above, where$(F+D)_{1}$ is an isomorphism, $U_{1}(N_{1})$ is finitely generated whereas $U_{2}^{-1}(N_{2})$ is \underline{not} finitely generated, it follows by  Lemma 2.16 that $$(F+D) \in \mathcal{M}\Phi_{+}(H_{\mathcal{A}}) \setminus \mathcal{M}\Phi(H_{\mathcal{A}})$$ (because, by that lemma, if $F+D$ was $\mathcal{A}$-Fredholm, then $U_{2}^{-1}(N_{2})$ would be finitely generated, which is a contradiction). The first part of the theorem follows, whereas the second part can be proved in the analogue way or can be deduced directly from the first part by passing to the adjoints and using  Corollary 2.11.
\end{proof}
\begin{corollary}
If $ F \in B^{a}(H_{\mathcal{A}})$ belongs to the boundary of $ \mathcal{M}\Phi(H_{\mathcal{A}})$ in $B^{a}(H_{\mathcal{A}})$ then $ F \notin \mathcal{M}\Phi_{\pm}(H_{\mathcal{A}}).$
\end{corollary}
\begin{proof}
Follows by the same arguments as in the proof of \cite[ Corollary 1.6.10] {ZZRD} since $$\mathcal{M}\Phi_{\pm} (H_{\mathcal{A}}) \setminus \mathcal{M}\Phi(H_{\mathcal{A}})= (\mathcal{M}\Phi_{+}(H_{\mathcal{A}})\setminus \mathcal{M}\Phi(H_{\mathcal{A}})) \cup (\mathcal{M}\Phi_{-}(H_{\mathcal{A}})\setminus \mathcal{M}\Phi(H_{\mathcal{A}})) $$ is open in $B^{a}(H_{\mathcal{A}}).$
\end{proof}
\begin{corollary}
Let $f:[0,1] \rightarrow B^{a}(H_{\mathcal{A}}) $ be continuous and assume that $f([0,1]) \subseteq \mathcal{M}\Phi_{\pm}(H_{\mathcal{A}}) .$ Then the following statments hold:\\
1) If $f(0) \in \mathcal{M}\Phi_{+}(H_{\mathcal{A}}) \setminus \mathcal{M}\Phi(H_{\mathcal{A}}) ,$ then $f(1) \in \mathcal{M}\Phi_{+}(H_{\mathcal{A}}) \setminus \mathcal{M}\Phi(H_{\mathcal{A}}) $\\
2) If $ f(0) \in \mathcal{M}\Phi_{-}(H_{\mathcal{A}}) \setminus \mathcal{M}\Phi(H_{\mathcal{A}}) ,$ then $f(1) \in \mathcal{M}\Phi_{-}(H_{\mathcal{A}}) \setminus \mathcal{M}\Phi(H_{\mathcal{A}}) $\\
3) If $f(0) \in \mathcal{M}\Phi(H_{\mathcal{A}}),$ then $f(1) \in \mathcal{M}\Phi(H_{\mathcal{A}}) $ and ${\rm index} f(0)={\rm index} f(1)$
\end{corollary}
\begin{proof}
We have that $\mathcal{M}\Phi_{\pm}(H_{\mathcal{A}}) $ is a disjoint union of $ \mathcal{M}\Phi_{+}(H_{\mathcal{A}}) \setminus \mathcal{M}\Phi(H_{\mathcal{A}}) ,$\\
$\mathcal{M}\Phi_{-}(H_{\mathcal{A}}) \setminus \mathcal{M}\Phi(H_{\mathcal{A}}) $ and $\mathcal{M}\Phi(H_{\mathcal{A}}) .$ The first two sets are open by previous theorem whereas $\mathcal{M}\Phi(H_{\mathcal{A}}) $ is open by \cite[ Lemma 2.7.10] {MT}. Moreover, by assumption in the corollary, we have that $f([0,1]) \subseteq \mathcal{M}\Phi_{\pm}(H_{\mathcal{A}}) .$ Since $f$ is continuous by assumption, $f([0,1])$ $ $ must be connected in $B^{a}(H_{\mathcal{A}}) ,$ hence $f([0,1]) $ must be completely contained in one of these three sets $$ \mathcal{M}\Phi_{+}(H_{\mathcal{A}}) \setminus \mathcal{M}\Phi(H_{\mathcal{A}}) , \mathcal{M}\Phi_{-}(H_{\mathcal{A}}) \setminus \mathcal{M}\Phi(H_{\mathcal{A}}) $$ or $\mathcal{M}\Phi(H_{\mathcal{A}}) $(otherwise we would get a separation of $f([0,1]) $ which is impossible). \\
Thus $1),2)$ and the first part of $3)$ folows. For the second part of $3),$ use \cite[ Lemma 2.7.10] {MT} together with the proof of \cite[ Lemma 1.6.1] {ZZRD}.
\end{proof}
\section{$\mathcal{M}\Phi_{+}^{-}$ and $\mathcal{M}\Phi_{-}^{+}$ operators on $H_{\mathcal{A}}$}
In this section we construct certain classes of operators on $H_{\mathcal{A}}$ as a generalizations of classes $ \Phi_{+}^{-}(H_{\mathcal{A}})$ , $ \Phi_{-}^{+}(H_{\mathcal{A}})$  (where $H$ is a Hilbert space). Then we investigate then and prove several properties concerning these new classes of operators on $H_{\mathcal{A}}$.
\begin{definition}
Let $F \in \mathcal{M}\Phi (H_{\mathcal{A}})$. 
We say that $F \in \tilde {\mathcal{M}} \Phi_{+}^{-} (H_{\mathcal{A}})$ if there exists a decomposition 
$$H_{\mathcal{A}} = M_{1} \tilde \oplus {N_{1}}‎‎\stackrel{F}{\longrightarrow} M_{2} \tilde \oplus N_{2}= H_{\mathcal{A}} $$
with respect to which $F$ has the matrix
\begin{center}
	$\left\lbrack
	\begin{array}{ll}
	F_{1} & 0 \\
	0 & F_{4} \\
	\end{array}
	\right \rbrack,
	$
\end{center}
where $F_{1}$ is an isomorphism, $N_{1},N_{2}$ are closed, finitely generated and $N_{1} \preceq N_{2},$ that is $N_{1}$ is isomorphic to a closed submodule of $N_{2}$. We define similarly the class $\tilde {\mathcal{M}}\Phi_{-}^{+} (H_{\mathcal{A}})$, the only difference in this case is that $N_{2} \preceq N_{1}$. Then we set
$$\mathcal{M}\Phi_{+}^{-} (H_{\mathcal{A}})= (\tilde {\mathcal{M}} \Phi_{+}^{-} (H_{\mathcal{A}})) \cup (\mathcal{M}\Phi_{+} (H_{\mathcal{A}}) \setminus \mathcal{M}\Phi (H_{\mathcal{A}}))$$ 
and
$$\mathcal{M}\Phi_{-}^{+} (H_{\mathcal{A}})= (\tilde {\mathcal{M}}\Phi_{-}^{+} (H_{\mathcal{A}})) \cup (\mathcal{M}\Phi_{-} (H_{\mathcal{A}}) \setminus \mathcal{M}\Phi (H_{\mathcal{A}}))$$ 
\end{definition}
\begin{lemma}
Suppose that $K(\mathcal{A})$ satisfies "the cancellation property". If $F \in \tilde{\mathcal{M}\Phi}_{+}^{-}(H_{\mathcal{A}}),$ then for any decomposition 
$$H_{\mathcal{A}} = M_{1}^{\prime} \tilde \oplus {N_{1}^{\prime}}‎‎\stackrel{F}{\longrightarrow} M_{2}^{\prime} \tilde \oplus N_{2}^{\prime}= H_{\mathcal{A}} $$ with respect to which $F$ has the matrix
\begin{center}
	$\left\lbrack
	\begin{array}{ll}
	F_{1}^{\prime} & 0 \\
	0 & F_{4}^{\prime} \\
	\end{array}
	\right \rbrack,
	$
\end{center} 
where $F_{1}^{\prime}$ is an isomorphism, $N_{1}^{\prime} , N_{2}^{\prime}$ are finitely generated, we have $N_{1}^{\prime} \preceq N_{2}^{\prime}.$ Similarly $N_{1}^{\prime} \preceq N_{2}^{\prime}$ if $F \in \tilde {\mathcal{M}}\Phi_{-}^{+} (H_{\mathcal{A}}).$
\end{lemma}
\begin{proof}
Given $F \in {\tilde {\mathcal{M}\Phi}}_{+}^{-} (H_{\mathcal{A}}),$ choose a decomposition for $F$\\ 
$$H_{\mathcal{A}} = M_{1} \tilde \oplus {N_{1} }‎‎\stackrel{F}{\longrightarrow} M_{2} \tilde \oplus N_{2}= H_{\mathcal{A}} $$
as described in the definition above. Then $N_{1}\cong N_{2,1} \preceq N_{2}$ for some closed submodule $N_{2,1}$ of $N_{2}$. Since $N_{1}$ is finitely generated, so is $N_{2,1}$, therefore, $N_{2,1}$, is orthogonally complementable in $N_{2}.$ So $N_{2}=N_{2,1} \oplus N_{2,2}$ for some closed submodule $N_{2,2} $ of $ N_{2}.$\\
Hence $${\rm index} F=[N_{1}]-[N_{2}]=[N_{2,1}]-[N_{2,1}]-[N_{2,2}]=-[N_{2,2}].$$
Thus
$${\rm index} F=[N_{1}^{\prime}]-[N_{2}^{\prime}]=-[N_{2,2}].$$
Taking the inverses on both sides of the equality in $K (A)$, we get $$[N_{2}^{\prime}]-[N_{1}^{\prime}]=[N_{2,2}],$$ so $$[N_{2}^{\prime}]=[N_{1}^{\prime}]+[N_{2,2}].$$\
Since $$[N_{1}^{\prime}]+[N_{2,2}]=[N_{1}^{\prime} \oplus N_{2,2}]=[N_{2}^{\prime}],$$ it follows that $$(N_{1}^{\prime} \oplus N_{2,2})\cong N_{2}^{\prime}$$ 
as 	$K(\mathcal{A})$ satisfies "the cancellation property".\\
Let $\tilde{\iota}:N_{1}^{\prime} \oplus N_{2,2}^{\prime} \longrightarrow N_{2}^{\prime} $ be the isomorphism, then since $N_{1}^{\prime} \oplus \lbrace 0\rbrace$ is a closed submodule of $N_{1}^{\prime} \oplus N_{2,2}$, it follows that $\tilde{\iota}(N_{1}^{\prime} \oplus \lbrace 0 \rbrace)$ is a closed submodule of $N_{2}^{\prime}$. Thus $(N_{1}^{\prime} \oplus \lbrace 0\rbrace ) \preceq N_{2}^{\prime} $. But $N_{1}^{\prime} \oplus \lbrace 0\rbrace \cong N_{1}^{\prime}$, so $N_{1}^{\prime} \preceq N_{2}^{\prime}$. One treats analogsly the case when $F \in \tilde {\mathcal{M}}\Phi_{-}^{+} (H_{\mathcal{A}}). $
\end{proof}
\begin{lemma}
$\tilde{\mathcal{M}\Phi}_{+}^{-}(H_{\mathcal{A}}) $ and $\tilde{\mathcal{M}\Phi}_{-}^{+}(H_{\mathcal{A}})$ are semigroups under multiplication.
\end{lemma}
\begin{proof}
Let $ F,D \in \tilde{\mathcal{M}\Phi}_{+}^{-}(H_{\mathcal{A}}) .$ Then there exist decompositions $$H_{\mathcal{A}} = M_{1} \tilde \oplus {N_{1}}‎‎\stackrel{F}{\longrightarrow} M_{2} \tilde \oplus N_{2}= H_{\mathcal{A}} $$ 
$$ H_{\mathcal{A}} = M_{1}^{\prime} \tilde \oplus {N_{1}^{\prime}}‎‎\stackrel{D}{\longrightarrow} M_{2}^{\prime} \tilde \oplus N_{2}^{\prime}= H_{\mathcal{A}} $$
with respect to which $F,D $ have matrices
$\left\lbrack
\begin{array}{ll}
F_{1} & 0 \\
0 & F_{4} \\
\end{array}
\right \rbrack ,
$
$\left\lbrack
\begin{array}{ll}
D_{1} & 0 \\
0 & D_{4} \\
\end{array}
\right \rbrack,
$
respectively, where $F_{1},D_{1} $ are isomorphisms, $N_{1},N_{2},N_{1}^{\prime},N_{2}^{\prime} $ are finitely generated and moreover $N_{1} \preceq N_{2}, N_{1}^{\prime} \preceq N_{2}^{\prime} .$ By the proof of \cite[ Lemma 2.7.11] {MT}, with respect to the decomposition 
$$H_{\mathcal{A}} = \overline{\overline{M_{1}}} \tilde \oplus \overline{\overline{N_{1}}}‎‎\stackrel{DF}{\longrightarrow} \overline{\overline{M_{2}^{\prime}}} \tilde \oplus \overline{\overline{N_{2}^{\prime}}} = H_{\mathcal{A}} $$ 
$DF$ has the matrix
$	\left\lbrack
\begin{array}{ll}
(DF)_{1} & 0 \\
0 & (DF)_{4} \\
\end{array}
\right \rbrack,
$
where $ (DF)_{1}$ is an isomorphism, $$\overline{\overline{N_{1}}}=U (F_{1}^{-1}(P)\tilde \oplus N_{1} ),\overline{\overline{N_{2}^{\prime}}}=D(P^{\prime}) \tilde \oplus N_{2}^{\prime}, (P \tilde \oplus N_{2} ) \cong (P \tilde \oplus N_{1}^{\prime} ) \cong L_{n}$$ for some $n$, $D_{\mid_{P}},F_{\mid_{P^{\prime}}} $ and $ U $ are isomorphisms. Since $N_{1} $ is isomorphic to a closed submodule of $N_{2}$ and $ F_{1}^{-1}(P) \cong P ,$ it follows that $ F_{1}^{-1}(P) \oplus N_{1} $ is isomorphic to a closed submodule of $P \oplus N_{2} $ (here we consider the direct sums of modules in the sense of \cite[ Example 1.3.4] {MT}). But since there are natural isomorphisms between ( $(F_{1}^{-1}(P)\tilde \oplus N_{1}) $) and ($ (F_{1}^{-1}(P) \oplus N_{1}) $), between ($ P \tilde \oplus N_{2} $) and ($ P \oplus N_{2} $), it follows that $ F_{1}^{-1}(P)\tilde \oplus N_{1} $ is isomorphic to a closed submodule of ($P \tilde \oplus N_{2} $). As $ U $ is an isomorphism, it follows that $\overline{\overline{N_{1}}}=U(F_{1}^{-1}(P)\tilde \oplus N_{1}) $ is isomorphic to a closed submodule of $ P \tilde \oplus N_{2} .$ Now, $ P \tilde \oplus N_{2} $ is isomorphic to $ P^{\prime} \tilde \oplus N_{1}^{\prime} ,$ so $\overline{\overline{N_{1}}} $ is isomorphic to a closed submodule of $ P^{\prime} \tilde \oplus N_{1}^{\prime} .$ Next, using that $ P^{\prime} \cong D(P^{\prime}) $ and that $ N_{1}^{\prime} $ is isomorphic to a closed submodule of $N_{2}^{\prime} ,$ by the same arguments as above (considering direct sums of modules), we can deduce that $ (P^{\prime} \tilde \oplus N_{1}^{\prime} ) $ is isomorphic to a closed submodule of
$(D(P^{\prime}) \tilde \oplus N_{2}^{\prime} ) = \overline{\overline{N_{2}^{\prime}}} ,$ so $ {\overline{\overline{N_{1}}}} \preceq ( P^{\prime} \tilde \oplus N_{1}^{\prime} ) \preceq \overline{\overline{N_{2}^{\prime}}}.$ Thus $DF \in \tilde{\mathcal{M}\Phi}_{-}^{+}(H_{\mathcal{A}}) .$ Similarly one can show that $\tilde{\mathcal{M}\Phi}_{-}^{+}(H_{\mathcal{A}})$ is a semigroup. 
\end{proof}
\begin{lemma}
$ \mathcal{M}\Phi_{+}^{-}(H_{\mathcal{A}}) $ and ${\mathcal{M}\Phi}_{-}^{+}(H_{\mathcal{A}})$ are semigroups under multiplication.
\end{lemma}
\begin{proof}
Let $ F,D \in \mathcal{M}\Phi_{+}^{-}(H_{\mathcal{A}}) .$ By definition, $ \mathcal{M}\Phi_{+}^{-}(H_{\mathcal{A}}) \subset \mathcal{M}\Phi_{+}(H_{\mathcal{A}}) ,$ so\\
$ F,D \in \mathcal{M}\Phi_{+}(H_{\mathcal{A}}) $ then. By  Corollary 2.5 $DF \in \mathcal{M}\Phi_{+}(H_{\mathcal{A}})   .$ Now, if $ F,D \in \tilde {\mathcal{M}}\Phi_{+}^{-}(H_{\mathcal{A}})  ,$ by the previous lemma it follows that $ DF \in \tilde {\mathcal{M}}\Phi_{+}^{-}(H_{\mathcal{A}}) .$ If $ D,F \in \mathcal{M}\Phi_{+}(H_{\mathcal{A}}) \setminus \mathcal{M}\Phi(H_{\mathcal{A}})  ,$ then $ DF \in \mathcal{M}\Phi_{+}(H_{\mathcal{A}}) \setminus \mathcal{M}\Phi(H_{\mathcal{A}}) $ as $\mathcal{M}\Phi_{+}(H_{\mathcal{A}}) \setminus \mathcal{M}\Phi(H_{\mathcal{A}}) $ is a semigroup by  Corollary 2.10.
If $F \in \mathcal{M}\Phi_{+}(H_{\mathcal{A}}) \setminus \mathcal{M}\Phi(H_{\mathcal{A}}) $ and $D \in \tilde {\mathcal{M}}\Phi_{+}^{-}(H_{\mathcal{A}}) ,$ then in particular $D \in \mathcal{M}\Phi(H_{\mathcal{A}}) $ as $\tilde {\mathcal{M}}\Phi_{+}^{-} (H_{\mathcal{A}}) \subseteq \mathcal{M}\Phi(H_{\mathcal{A}}) $ by definition. By  Corollary 2.9, it follows that $ DF $ can \underline{not} be in $\mathcal{M}\Phi(H_{\mathcal{A}}) $ as $F \notin \mathcal{M}\Phi(H_{\mathcal{A}}) .$ Since $DF \in \mathcal{M}\Phi_{+}(H_{\mathcal{A}}) ,$ we get that $ DF \in \mathcal{M}\Phi_{+}(H_{\mathcal{A}}) \setminus \mathcal{M}\Phi(H_{\mathcal{A}}) .$
If $ D \in \mathcal{M}\Phi_{+}(H_{\mathcal{A}}) \setminus \mathcal{M}\Phi(H_{\mathcal{A}}) ,$ it is clear that $DF$ can not be an element of $ \mathcal{M}\Phi(H_{\mathcal{A}})  .$ Indeed, if $ DF \in \mathcal{M}\Phi(H_{\mathcal{A}}) $ then by  Corollary 2.6 we would get that $D \in \mathcal{M}\Phi_{-}(H_{\mathcal{A}})  .$ Hence $D \in \mathcal{M}\Phi_{-}(H_{\mathcal{A}}) \cap \mathcal{M}\Phi_{+}(H_{\mathcal{A}} )) $ which is a contradiction as $\mathcal{M}\Phi_{-}(H_{\mathcal{A}}) \cap \mathcal{M}\Phi_{+}(H_{\mathcal{A}}) = \mathcal{M}\Phi(H_{\mathcal{A}}) $ by  Corollary 2.4. Collecting all these arguments together, we deduce that $\mathcal{M}\Phi_{+}^{-}(H_{\mathcal{A}}) $ is a semigroup. Similarly one can show that $\mathcal{M}\Phi_{-}^{+}(H_{\mathcal{A}}) $ is a semigroup.
\end{proof}
\begin{lemma}
$\tilde {\mathcal{M}}\Phi_{+}^{-}(H_{\mathcal{A}}) $ and $\tilde {\mathcal{M}}\Phi_{-}^{+}(H_{\mathcal{A}}) $ are open.
\end{lemma}
\begin{proof}
Given $F \in \tilde {\mathcal{M}}\Phi_{+}^{-}(H_{\mathcal{A}}) $,
let 
$$H_{\mathcal{A}} = M_{1} \tilde \oplus {N_{1}}‎‎\stackrel{F}{\longrightarrow} M_{2} \tilde \oplus N_{2}= H_{\mathcal{A}} $$ be a decomposition, with respect to which
\begin{center}
	$F=\left\lbrack
	\begin{array}{ll}
	F_{1} & 0 \\
	0 & F_{4} \\
	\end{array}
	\right \rbrack,
	$
\end{center}
where $F_{1}$ is an isomorphism, $N_{1},N_{2}$ are finitely generated and $N_{1}\preceq N_{2}$. By the proof of \cite[ Lemma 2.7.10] {MT}, there exists an $\epsilon >0$ s.t. if $ \parallel F-D \parallel < \epsilon$, then there exists a decomposition $$H_{\mathcal{A}} = M_{1}^{\prime} \tilde \oplus {N_{1}^{\prime}}‎‎\stackrel{D}{\longrightarrow} M_{2}^{\prime} \tilde \oplus N_{2}^{\prime}= H_{\mathcal{A}} $$ with respect to which
\begin{center}
	$D=\left\lbrack
	\begin{array}{ll}
	D_{1} & 0 \\
	0 & D_{4} \\
	\end{array}
	\right \rbrack,
	$
\end{center}
where $D_{1}$ is an isomorphism, and moreover $$M_{1} \cong M_{1}^{\prime}, N_{1} \cong N_{1}^{\prime},M_{2} \cong M_{2}^{\prime} \textrm{ and } N_{2} \cong N_{2}^{\prime}. $$ 
Let $$U_{1}:N_{1}^{\prime} \rightarrow N_{1}, U_{2}:N_{2} \rightarrow N_{2}^{\prime}$$ be these isomorphisms. Since $N_{1}\preceq N_{2}$, there exists an isomorphism $\tilde{\iota}$ from $N_{1}$ onto some closed submodule $\tilde {\iota}(N_{1})\subseteq N_{2}$. Then $U_{2} \tilde {\iota} U_{1}$ is an isomorphism from $N_{1}^{\prime}$ onto $(U_{2} \tilde {\iota} U_{1})(N_{1})$ which is a closed submodule of $N_{2}^{\prime}.$\\
Thus $N_{1}^{\prime} \preceq N_{2}^{\prime}$ ( and also $N_{1}^{\prime} , N_{2}^{\prime} $ are finitely generated as $N_{1} , N_{2}$ are so). Therefore, $D \in \mathcal{M}\Phi_{+}^{-} (H_{\mathcal{A}})$. Similarly we can show that $\tilde{\mathcal{M}\Phi}_{-}^{+} (H_{\mathcal{A}})$ is open.
\end{proof}
\begin{definition}
Let $F \in \mathcal{M}\Phi_{+} (H_{\mathcal{A}}).$ We say that $ F \in {{\mathcal{M}\Phi}_{+}^{-}}^{\prime} (H_{\mathcal{A}})$ if there exists a decomposition $$H_{\mathcal{A}} = M_{1} \tilde \oplus {N_{1}}‎‎\stackrel{F}{\longrightarrow} M_{2} \tilde \oplus N_{2}= H_{\mathcal{A}} $$
with respect to which
\begin{center}
	$F=\left\lbrack
	\begin{array}{ll}
	F_{1} & 0 \\
	0 & F_{4} \\
	\end{array}
	\right \rbrack,
	$
\end{center}
where $F_{1}$ is an isomorphism, $N_{1}$ is closed, finitely generated and $N_{1} \preceq N_{2} .$ Similarly, we define the class ${\mathcal{M}\Phi_{-}^{+}}^{\prime} (H_{\mathcal{A}})$, only in this case $F \in \mathcal{M}\Phi_{-} (H_{\mathcal{A}})$, $N_{2}$ is finitely generated and $N_{2} \preceq N_{1} .$
\end{definition}
\begin{proposition}
$$\tilde {\mathcal{M}}\Phi_{+}^{-}(H_{\mathcal{A}}) = { \mathcal{M}\Phi_{+}^{-}}^{\prime}(H_{\mathcal{A}}) \cap \mathcal{M}\Phi(H_{\mathcal{A}}) , \tilde {\mathcal{M}}\Phi_{-}^{+}(H_{\mathcal{A}}) = { \mathcal{M}\Phi_{-}^{+}}^{\prime}(H_{\mathcal{A}}) \cap \mathcal{M}\Phi(H_{\mathcal{A}}) $$
\end{proposition}
\begin{proof}
By definition of $\tilde {\mathcal{M}}\Phi_{+}^{-}(H_{\mathcal{A}})  ,$ the inclusion $"\subseteq " $ is obvious. Let us show the other inclusion. To this end, choose some 
${D \in {{\mathcal{M}\Phi}_{+}^{-}}}^{\prime} (H_{\mathcal{A}}) \cap \mathcal{M}\Phi(H_{\mathcal{A}}) .$ Since ${D \in {{\mathcal{M}\Phi}_{+}^{-}}}^{\prime} (H_{\mathcal{A}}) ,$ there exists a decomposition 
$$ H_{\mathcal{A}} = M_{1}^{\prime} \tilde \oplus {N_{1}^{\prime}}‎‎\stackrel{D}{\longrightarrow} M_{2}^{\prime} \tilde \oplus N_{2}^{\prime}= H_{\mathcal{A}} $$ 
with respect to which $D$ has the matrix
$\left\lbrack
\begin{array}{ll}
D_{1} & 0 \\
0 & D_{4} \\
\end{array}
\right \rbrack,
$
where $ D_{1} $ is an isomorphism, $ N_{1}^{\prime} $ is finitely generated and $ N_{1}^{\prime} \preceq N_{2}^{\prime}  .$ On the other hand, since $D \in \mathcal{M}\Phi (H_{\mathcal{A}}) ,$ there exists a decomposition
$$ H_{\mathcal{A}} = M_{1} \tilde \oplus {N_{1}}‎‎\stackrel{D}{\longrightarrow} M_{2} \tilde \oplus N_{2}= H_{\mathcal{A}} $$
with respect to which
$D=\left\lbrack
\begin{array}{ll}
\tilde D_{1} & 0 \\
0 & \tilde D_{4} \\
\end{array}
\right \rbrack,
$ 
where $\tilde D_{1} $ is an isomorphism, $ N_{1}, N_{2} $ are finitely generated. By  Lemma 2.16, $N_{2}^{\prime} $ must be then finitely generated. Hence ${D \in {\tilde{\mathcal{M}\Phi}_{+}^{-}}}^{\prime} (H_{\mathcal{A}}) .$ Similarly, using  Lemma 2.17, one can show that $$\tilde {\mathcal{M}}\Phi_{-}^{+}(H_{\mathcal{A}}) = { \mathcal{M}\Phi_{-}^{+}}^{\prime}(H_{\mathcal{A}}) \cap \mathcal{M}\Phi(H_{\mathcal{A}}). $$
\end{proof}
\begin{remark}
${\mathcal{M}\Phi_{+}^{-}}^{\prime} \subseteq {\mathcal{M}\Phi_{+}^{-}}$ and ${\mathcal{M}\Phi_{-}^{+}}^{\prime} \subseteq {\mathcal{M}\Phi_{-}^{+}},$ and on Hilbert spaces "=" holds due to that given any finite dimensional subspace $N_{1}$ and infinite dimensional subspace $N_{2}$, then $N_{1}$ is isomorphic to a closed subspace of $N_{2}.$
\end{remark}
\begin{lemma}
The sets ${\mathcal{M}\Phi_{-}^{+}}^{\prime} (H_{\mathcal{A}})$ and ${\mathcal{M}\Phi_{+}^{-}}^{\prime} (H_{\mathcal{A}})$ are open. Moreover, if $F \in {\mathcal{M}\Phi_{+}^{-}}^{\prime}(H_{\mathcal{A}}) $ and $K \in K(H_{\mathcal{A}})$, then $$(F+K) \in {\mathcal{M}\Phi_{+}^{-}}^{\prime} (H_{\mathcal{A}}).$$ 
If $ F \in {\mathcal{M}\Phi_{-}^{+}}^{\prime}(H_{\mathcal{A}}) $ and $K \in K(H_{\mathcal{A}})$, then $$(F+K) \in {\mathcal{M}\Phi_{-}^{+}}^{\prime} (H_{\mathcal{A}}).$$ 
\end{lemma}
\begin{proof}
Suppose $F \in {\mathcal{M}\Phi_{+}^{-}}^{\prime} (H_{\mathcal{A}})$ and choose a decomposition.
$$H_{\mathcal{A}} = M_{1} \tilde \oplus {N_{1}}‎‎\stackrel{F}{\longrightarrow} M_{2} \tilde \oplus N_{2}= H_{\mathcal{A}} $$ such that $N_{1} \preceq N_{2} $ as described in the Definition 5.6. Then, again by the proof of \cite[ Lemma 2.7.10] {MT}, we have that there exists an $\epsilon >0$ such that if $\parallel F-D\parallel <\epsilon,$ then there exists a decomposition 
$$H_{\mathcal{A}} = M_{1}^{\prime} \tilde \oplus {N_{1}^{\prime}}‎‎\stackrel{D}{\longrightarrow} M_{2}^{\prime} \tilde \oplus N_{2}^{\prime}= H_{\mathcal{A}} $$ with respect to which $D$ has the matrix
\begin{center}
	$\left\lbrack
	\begin{array}{ll}
	D_{1} & 0 \\
	0 & D_{4} \\
	\end{array}
	\right \rbrack
	$
\end{center}
where $D_{1}$ is an isomorphism and $N_{1}^{\prime} \cong N_{1},N_{2}^{\prime} \cong N_{2}$. Therefore, by the same arguments as in the proof of  Lemma 5.5, we have $N_{1}^{\prime} \preceq N_{2}^{\prime} $ as $N_{1} \preceq N_{2} $. Thus $D$ is in ${\mathcal{M}\Phi_{+}^{-} (H_{\mathcal{A}})}^{\prime}$ also, so ${\mathcal{M}\Phi_{+}^{-} (H_{\mathcal{A}})}^{\prime}$ is open. Next, let $K \in K(H_{\mathcal{A}})$. By the proof of \cite[ Lemma 2.7.13] {MT} there exists an $L_{n} $ such that $F+K$ has the matrix
\begin{center}
	$\left\lbrack
	\begin{array}{ll}
	(F+K)_{1} & 0 \\
	0 & (F+K)_{4} \\
	\end{array}
	\right \rbrack
	$
\end{center}
with respect to the decomposition 
$$H_{\mathcal{A}} = U_{1}^{\prime}(L_{n}^{\bot}) \tilde \oplus {U_{1}^{\prime}(L_{n})}‎‎\stackrel{F+K}{\longrightarrow} {U_{2}^{\prime}}^{-1}(FL_{n}^{\bot}) \tilde \oplus {U_{2}^{\prime}}^{-1}(F(P) \tilde \oplus N_{2})= H_{\mathcal{A}},$$ 
where $(F+K)_{1}$ is an isomorphism, $L_{n}=N_{1} \tilde \oplus P$, $P=M_{1} \cap L_{n} $ and $P\cong F(P)$ for some closed, finitely generated submodule $P$ (here $F,N_{1},N_{2} $ are as given above). Now, since $N_{1}$ is isomorphic to a closed submodule of $N_{2},$ then clearly $ P \tilde \oplus N_{1}$ is isomorphic to a closed submodule of $ F(P) \tilde \oplus N_{2}$ as $P \cong F(P)$. Therefore, $$(P \tilde \oplus N_{1}) \preceq (F(P)\tilde \oplus N_{2}) .$$
Since $U_{1}^{\prime},U_{2}^{\prime}$ are isomorphisms, then 
$$U_{1}^{\prime}(L_{n}) = U_{1}^{\prime} (P \tilde \oplus N_{1}) \preceq {U_{2}^{\prime}}^{-1}(F(P) \tilde \oplus N_{2}), $$ 
so $(F+K) \in {\mathcal{M}\Phi_{+}^{-}}^{\prime} (H_{\mathcal{A}})$. Similarly one proves the statments for ${\mathcal{M}\Phi_{-}^{+}}^{\prime} (H_{\mathcal{A}}).$ 
\end{proof} 
\begin{theorem}
Let $F \in B^{a}(H_{\mathcal{A}})$. The following statements are equivalent\\
1)	$F \in {\mathcal{M}\Phi_{+}^{-}}^{\prime} (H_{\mathcal{A}})$\\
2)	There exist $D \in B^{a}(H_{\mathcal{A}}), K \in K(H_{\mathcal{A}})$ such that $D$ is bounded below and $F=D+K$
\end{theorem}
\begin{proof}
$1) \rightarrow 2$)\\
Let $F \in {\mathcal{M}\Phi_{+}^{-}}^{\prime} (H_{\mathcal{A}})$ and let
$$H_{\mathcal{A}} = M_{1} \tilde \oplus {N_{1}}‎‎\stackrel{F}{\longrightarrow} M_{2} \tilde \oplus N_{2}= H_{\mathcal{A}} $$ 
be a decomposition as given in the Definition 5.6, so that $N_{1}$ is finitely generated, $N_{1} \preceq N_{2},$ and $F_{\mid_{M_{1}}}\rightarrow M_{2}$ is an isomorphism. Since $ N_{1}$ is finitely generated, by the proof of \cite[ Theorem 2.7.6] {MT}, we may assume that $ M_{1}=N_{1}^{\bot}.$ Let $\iota$ be the isomorphism from $N_{1}$ onto a closed submodule $\iota(N_{1}) \subseteq N_{2}.$ Set $D=F+(\iota-F)P_{N_{1}}$, where $P_{N_{1}}$ is the orthogonal projection onto $N_{1}.$ Then $(\iota-F)P_{N_{1}}$ is in $K (H_{\mathcal{A}})$ and in addition $D = F+ (\iota-F)P_{N_{1}}=FP_{M_{1}}+\iota P_{N_{1}}.$ Since $F_{{\mid}_{{M}_{1}}}$ is an isomorphism from $M_{1}$ onto $M_{2}$, $\iota$ is an isomorphism from $N_{1}$ onto $\iota(N_{1})\subseteq N_{2}$ and $H_{\mathcal{A}} = M_{2} \tilde \oplus N_{2},$ it follows that $D$ is bounded below being an isomorphism of $H_\mathcal{A} $ onto $M_{2} \tilde \oplus \iota (N_{1})$ which is a closed submodule of $H_{\mathcal{A}}.$ Moreover $F=D + (F-\iota)P_{N_{1}}$ and $(F-\iota)P_{N_{1}}$ is compact. Note that $\iota{P_{N_{1}}}$ is indeed adjointable: Since $\iota :N_{1}\rightarrow \iota (N_{1})\subseteq N_{2}$ and $N_{1}$ is self-dual being finitely generated, then by \cite[ Proposition 2.5.2] {MT}, the result which was originally proved in \cite{P}, $\iota$ is adjointable. Moreover, since $\iota (N_{1})$ is finitely generated being isomorphic to $N_{1},$ it follows that $ \iota (N_{1})$ is an orthogonal direct summand in $H_{\mathcal{A}}$ by \cite[ Lemma 2.3.7] {MT}. Hence the inclusion $J_{\iota (N_{1})}:\iota (N_{1})\rightarrow H_{\mathcal{A}}$ is adjointable. Also $P_{N_{1}}$ is adjointable, so $\iota P_{N_{1}}=J_{\iota (N_{1})}\iota P_{N_{1}} \in B^{a}(H_{\mathcal{A}}).$\\
$2)\Rightarrow 1)$\\
If $D \in B^{a}(H_{\mathcal{A}})$ is bounded below, then obviously $D \in {\mathcal{M}\Phi_{+}^{-}}^{\prime}(H_{\mathcal{A}}) .$ As $K \in K(\mathcal{A}),$ by previous lemma we get that $(D+K) \in {\mathcal{M}\Phi_{+}^{-}}^{\prime}(H_{\mathcal{A}}). $
\end{proof}
\begin{proposition}
1)$F \in {\mathcal{M}\Phi_{+}^{-}}^{\prime} (H_{\mathcal{A}}) \Leftrightarrow F^{*} \in {\mathcal{M}\Phi_{-}^{+}}^{\prime} (H_{\mathcal{A}})$\\
2) $ F \in {{\tilde{\mathcal{M}\Phi}}_{+}^{-}} (H_{\mathcal{A}}) \Leftrightarrow {F}^{*} \in {\tilde{\mathcal{M}\Phi}}_{-}^{+} (H_{\mathcal{A}}) $\\
3) $ F \in {\mathcal{M}\Phi_{+}^{-}} (H_{\mathcal{A}}) \Leftrightarrow F^{*} \in {\mathcal{M}\Phi_{-}^{+}} (H_{\mathcal{A}}) $
\end{proposition}
\begin{proof}
1)	Let $F \in {\mathcal{M}\Phi_{-}^{+}}^{\prime} (H_{\mathcal{A}})$, choose a decomposition $$H_{\mathcal{A}} = M_{1} \tilde \oplus N_{1} \stackrel{F}{\longrightarrow} M_{2} \tilde \oplus N_{2}= H_{\mathcal{A}} $$ 
w.r.t which $F$ has the matrix 
\begin{center}
	$\left\lbrack
	\begin{array}{ll}
	F_{1} & 0 \\
	0 & F_{4} \\
	\end{array}
	\right \rbrack
	,$
\end{center} 
where $F_{1}$ is an isomorphism, $N_{1} \preceq N_{2}$ and $N_{1}$ is finitely generated. Again, by the proof of \cite[Theorem 2.7.6] {MT}, we may assume that $M_{1}=N_{1}^{\bot}. $ W.r.t. the decomposition
$$H_{\mathcal{A}} = {N_{1}}^{\bot} \tilde \oplus {N_{1}}‎\stackrel{F}{\longrightarrow} F(N_{1}^{\bot})  \oplus F(N_{1}^{\bot})^{\bot}=H_{\mathcal{A}},$$ 
$F$ has the matrix
\begin{center}
	$\left\lbrack
	\begin{array}{ll}
	\tilde F_{1} & \tilde F_{2} \\
	0 &  \tilde F_{4} \\
	\end{array}
	\right \rbrack
	,$
\end{center}
where $\tilde F_{1}$ is an isomorphism and $\tilde F_{1},\tilde F_{2} ,\tilde F_{4} $ are adjointable, so
\begin{center}
	$F^{*}=\left\lbrack
	\begin{array}{ll}
	\tilde F_{1}^{*} & 0 \\
	\tilde F_{2}^{*} &  \tilde F_{4}^{*} \\
	\end{array}
	\right \rbrack
	$
\end{center}	
w.r.t. the decomposition $$H_{\mathcal{A}} = {F(N_{1}^{\bot})  \oplus F(N_{1}^{\bot})^{\bot}}‎\stackrel{F^{*}}{\longrightarrow} {N_{1}}^{\bot} \tilde \oplus N_{1}=H_{\mathcal{A}}.$$  
Moreover, since $\tilde F_{1}^{*} $ is an isomorphism, $ F^{*} $ has the matrix
\begin{center}
	$\left\lbrack
	\begin{array}{ll}
	\tilde  {\tilde F_{1}^{*}} & 0 \\
	0 &  \tilde F_{4}^{*} \\
	\end{array}
	\right \rbrack
	$
\end{center} 
w.r.t. the decomposition 
$$H_{\mathcal{A}} = {N_{2}}^{\bot} \tilde \oplus N_{2} ‎\stackrel{F^{*}}{\longrightarrow} V^{-1}({N_{1}}^{\bot}) \tilde \oplus V^{-1}(N_{1})=H_{\mathcal{A}},$$  
where
\begin{center}
	$V=\left\lbrack
	\begin{array}{ll}
	1 & 0 \\
	- \tilde F_{2}^{*}(\tilde F_{1}^{*})^{-1} &  1 \\
	\end{array}
	\right \rbrack
	$
\end{center} 
w.r.t the decomposition 
$$H_{\mathcal{A}} = {N_{1}}^{\bot} \oplus N_{1} ‎\stackrel{V}{\longrightarrow} {N_{1}}^{\bot} \oplus N_{1}=H_{\mathcal{A}},$$  
so that $V$ is an isomorphism and also $\tilde{\tilde{F}}_{1}^{*}$ is  an isomorphism. Now, since $V$ is an isomorphism and there exists an isomorphism  $\iota :N_{1}\rightarrow \iota(N_{1})\subseteq N_{2}$  (as $N_{1} \preceq N_{2} $), we get that $\iota V:V^{-1}(N_{1})\rightarrow \iota (N_{1})\subseteq N_{2}$ is an isomorphism, so $V^{-1}(N_{1}) \preceq N_{2}$. Moreover, $V^{-1}(N_{1})$ finitely generated, as $N_{1}$ is so. Therefore, $F^{*} \in {\mathcal{M}\Phi_{-}^{+}}^{\prime} (H_{\mathcal{A}})$. Conversely, if $F \in {\mathcal{M}\Phi_{-}^{+}}^{\prime} (H_{\mathcal{A}}),$ let 
$$H_{\mathcal{A}} = M_{1} \tilde \oplus N_{1} ‎\stackrel{F}{\longrightarrow} M_{2} \tilde \oplus N_{2}= H_{\mathcal{A}} $$
be an $ {\mathcal{M}\Phi_{-}^{+}}^{\prime}$ decomposition for $F$, then
$N_{2} \preceq N_{1} $ and $N_{2}$ is finitely generated. By the proof of Theorem 2.3 part $1) \Rightarrow 2)$ $F$  has the matrix
\begin{center}
	$\left\lbrack
	\begin{array}{ll}
	\tilde F_{1} & 0 \\
	\tilde F_{3} &  \tilde F_{4} \\
	\end{array}
	\right \rbrack
	$
\end{center}
w.r.t. the decomposition 
$$H_{\mathcal{A}} = {N_{1}}^{\bot} \tilde \oplus {N_{1}} ‎\stackrel{F}{\longrightarrow}  {N_{2}}^{\bot} \tilde \oplus N_{2}=H_{\mathcal{A}} ,$$ 
where $\tilde F_{1} ,\tilde F_{3} ,\tilde F_{4} $  are adjointable and $\tilde F_{1}$ is an isomorphism. Then $F^{*}$ has the matrix
\begin{center}
	$\left\lbrack
	\begin{array}{ll}
	\tilde F_{1}^{*} & \tilde F_{2}^{*}	 \\
	0 &\tilde F_{4}^{*} \\
	\end{array}
	\right \rbrack
	$
\end{center} 
w.r.t. the decomposition  
$$H_{\mathcal{A}} = {N_{2}}^{\bot} \tilde \oplus {N_{2}} ‎\stackrel{F^{*}} {\longrightarrow} {N_{1}}^{\bot} \tilde \oplus N_{1}=H_{\mathcal{A}} ,$$
and $\tilde F_{1}^{*}$ is an isomorphism. Hence
\begin{center}
	$F^{*}=\left\lbrack
	\begin{array}{ll}
	\tilde F_{1}^{*} & 0 \\
	0 &\tilde {\tilde F_{4}^{*}} \\
	\end{array}
	\right \rbrack
	$
\end{center} 
w.r.t. the decomposition $$H_{\mathcal{A}} =U({N_{2}}^{\bot}) \tilde \oplus U(N_{2})‎\stackrel{F^{*}}{\longrightarrow} {N_{1}}^{\bot} \tilde \oplus N_{1} =H_{\mathcal{A}},$$   where
\begin{center}
	$U=\left\lbrack
	\begin{array}{ll}
	1 & - \tilde {F_{1}^{*}}^{-1}(\tilde F_{3}^{*}) \\
	0 &  1 \\
	\end{array}
	\right \rbrack
	$
\end{center} 
w.r.t the decomposition 
$$H_{\mathcal{A}} = {N_{2}}^{\bot} \oplus N_{2} ‎\stackrel{V}{\longrightarrow} {N_{2}}^{\bot} \oplus N_{2}=H_{\mathcal{A}},$$ 
so that $U$ is an isomorphism. Since  $\iota:N_{2}\Rightarrow \iota (N_{2})\subseteq N_{1}$ is an isomorphism, then $\iota U^{-1}:U(N_{2})\rightarrow \iota (N_{2})\subseteq N_{1}$ is also an isomorphism, so $U(N_{2})\preceq N_{1}$. Thus $F^{*} \in {\mathcal{M}\Phi_{+}^{-}}^{\prime} (H_{\mathcal{A}}).$
\\
2)	Use 1) together with the fact that $$F \in \mathcal{M}\Phi (H_{\mathcal{A}}) \Leftrightarrow F^{*} \in \mathcal{M}\Phi (H_{\mathcal{A}}) $$ by Corollary 2.11 and the fact that
$$\tilde {\mathcal{M}} \Phi_{-}^{+} (H_{\mathcal{A}})={\mathcal{M}\Phi_{-}^{+}}^{\prime} (H_{\mathcal{A}}) \cap \mathcal{M}\Phi (H_{\mathcal{A}})$$
$$\tilde {\mathcal{M}} \Phi_{+}^{-} (H_{\mathcal{A}})={\mathcal{M}\Phi_{+}^{-}}^{\prime} (H_{\mathcal{A}}) \cap \mathcal{M}\Phi (H_{\mathcal{A}})$$
by Proposition 5.7.\\
3)	Use 2) together with the fact that
$$F \in  \mathcal{M}\Phi_{+} (H_{\mathcal{A}}) \setminus \mathcal{M}\Phi (H_{\mathcal{A}}))\Leftrightarrow F^{*} \in  \mathcal{M}\Phi_{-} (H_{\mathcal{A}}) \setminus \mathcal{M}\Phi (H_{\mathcal{A}}))$$ 
by Corollary 2.11 and the fact that 
$${\mathcal{M}\Phi_{+}^{-}} (H_{\mathcal{A}})= {\tilde {\mathcal{M}}\Phi_{+}^{-}} (H_{\mathcal{A}}) \cup (\mathcal{M}\Phi_{+} (H_{\mathcal{A}}) \setminus \mathcal{M}\Phi (H_{\mathcal{A}}))),$$
$${\mathcal{M}\Phi_{-}^{+}} (H_{\mathcal{A}})= {\tilde {\mathcal{M}}\Phi_{-}^{+}} (H_{\mathcal{A}}) \cup (\mathcal{M}\Phi_{-} (H_{\mathcal{A}}) \setminus \mathcal{M}\Phi (H_{\mathcal{A}})))$$ by Definition 5.1.
\end{proof}
\begin{definition}
We set $M^{a}(H_{\mathcal{A}})= \lbrace F \in B^{a}(H_{\mathcal{A}}) \mid F $ is bounded below.$\rbrace$
\end{definition}
\begin{lemma}
Let $Q^{a} (H_{\mathcal{A}})=\lbrace D \in B^{a}(H_{\mathcal{A}}) \mid D \textrm{ is surjective } \rbrace  .$ Then\\
$ F \in M^{a}(H_{\mathcal{A}}) $ if and only if $ F^{*} \in Q^{a} (H_{\mathcal{A}}) .$
\end{lemma}
\begin{proof}
Let $ F \in M^{a}(H_{\mathcal{A}}) .$ By the proof of \cite[ Theorem 2.3.3] {MT}, as ${\rm ran} F$ is closed in this case, we have that ${\rm ran} F^{*} $ is also closed. Moreover, by the proof of \cite[ Theorem 2.3.3] {MT} since ${\rm ran} F^{*}$ is closed, we also have $H_{\mathcal{A}}= \ker F \oplus {\rm ran} F^{*}$. Since $\ker F = \lbrace 0 \rbrace  ,$ it follows that $H_{\mathcal{A}}={\rm ran} F^{*}. $\\
Conversely, if $ F^{*} \in Q^{a} (H_{\mathcal{A}})  ,$ then $\ker F={{\rm ran} F^{*}}^{\perp} = \lbrace 0 \rbrace   ,$ so $F$ is injective. Moreover, since $ {\rm ran} F^{*}= H_{\mathcal{A}} $ is closed, then ${\rm ran} F$ is closed also, (again by the proof of \cite[ Theorem 2.3.3] {MT}). By the Banach open mapping theorem, it follows that $F$ is an isomorphism from $ H_{\mathcal{A}} $ onto its image. Thus $F$ is bounded below.
\end{proof}
\begin{corollary}
Let $ D \in B^{a}(H_{\mathcal{A}})  .$ The following statements are equivalent:\\
1)	$D \in {\mathcal{M}\Phi_{-}^{+}}^{\prime}(H_{\mathcal{A}}) $\\
2)	There exist $ Q \in Q^{a} (H_{\mathcal{A}}) , K \in K(H_{\mathcal{A}}) $ s.t. $ D=Q+K .$
\end{corollary}
\begin{proof}
Follows from Theorem 5.10, Proposition 5.11 part 1) and  Lemma 5.13 by passing to the adjoints. 
\end{proof}
	\dedicatory{\textbf{Acknowlegedgement.} I am especially grateful to my main supervisor Professor Vladimir M. Manuilov for careful reading of the first version of this paper and for detailed comments and suggestions which led to the improved presentation of the paper.
	I am also grateful to my supervisor Professor Dragan S. \DJ{}or\dj{}evi\'{c} for suggesting the research topic of this paper and for introducing to me the relevant reference books. In addition, I am also grateful to my external supervisor Professor Camillo Trapani for reading my paper and for inspiring and useful comments and advice. Finally, I am grateful to the Refere for careful reading of my paper and for the inspiring and useful comments and suggestions that led to the final, improved presentation of the paper.
}
\bibliographystyle{amsplain}

\end{document}